\documentclass[11pt]{article}

\usepackage{amsthm, amssymb, bm, amsmath,cancel,mathrsfs}
\usepackage{soul, color,verbatim}
\usepackage[pagebackref,colorlinks=true,pdfpagemode=none,urlcolor=blue,linkcolor=blue,citecolor=blue]{hyperref}

\usepackage[colorinlistoftodos,disable]{todonotes}

\newtheorem{theorem}{Theorem}[]
\newtheorem{lemma}[theorem]{Lemma}
\newtheorem{proposition}[theorem]{Proposition}
\newtheorem{corollary}[theorem]{Corollary}

\newtheorem{remark}[theorem]{Remark}

\newcommand \Cm { \mathbb{C}}
\newcommand \Dm { \mathbb{D}}
\newcommand \Rm { \mathbb{R}}
\newcommand \Nm { \mathbb{N}}
\newcommand \Sm { \mathbb{S}}
\newcommand \Zm { \mathbb{Z}}

\renewcommand \L {{\cal L}}
\newcommand \A {{\cal A}}
\newcommand \B {{\cal B}}

\newcommand \G {{\cal G}}
\newcommand \N {{\cal N}}
\renewcommand \H {{\cal H}}
\newcommand \ft {{\mathfrak t}}

\newcommand \fN { {\mathcal N }}
\newcommand \dom {{\text{dom }}}
\newcommand \ran {{\text{ran }}}
\renewcommand \P {P}

\newcommand \bg {\bar{g}}

\newcommand \wtH {\widetilde{H}}

\newcommand \response[1] {{#1}}

\renewcommand \d { {\mathrm{d} }}

\newcommand \dprod[2] {\left( #1, #2\right)}

\newcommand \dir {\tau_\gamma^D}
\newcommand \neu {\tau_\gamma^N}
\newcommand \dirneu {\tau_\gamma^{D,N}}

\newcommand \ngamma {n_\gamma}

\title{Boundary triples for a family of degenerate elliptic operators of Keldysh type}
\author{Fran\c{c}ois Monard\thanks{Department of Mathematics, University of California, Santa Cruz CA 95064; email:fmonard@ucsc.edu} \and Yuzhou Zou\thanks{Department of Mathematics, Northwestern University, Evanston IL 60208; email:yuzhou.zou@northwestern.edu}}
\date{}

\begin{document}
\maketitle

\begin{abstract}
    We consider a one-parameter family of degenerately elliptic operators $\L_\gamma$ on the closed disk $\Dm$, of Keldysh (or Kimura) type, which appears in prior work \cite{Mishra2022} by the authors and Mishra, related to the geodesic X-ray transform. Depending on the value of a constant $\gamma\in \Rm$ in the sub-principal term, we prove that either the minimal operator is self-adjoint (case $|\gamma|\ge 1$), or that one may construct appropriate trace maps and Sobolev scales (on $\Dm$ and $\Sm^1=\partial\Dm$) on which to formulate mapping properties, Dirichlet-to-Neumann maps, and extend Green's identities (case $|\gamma|<1$). The latter can be reinterpreted in terms of a boundary triple for the maximal operator, or a generalized boundary triple for a distinguished restriction of it. The latter concepts, object of interest in their own right, provide avenues to describe sufficient conditions for self-adjointness of extensions of $\L_{\gamma,min}$ that are parameterized in terms of boundary relations, and we formulate some corollaries to that effect.  
\end{abstract}

\tableofcontents

\section{Introduction} \label{sec:intro}

This article is concerned with the study of boundary triples (or equivalently, the derivation of appropriate settings where generalized Green's identities\footnote{For the Laplacian on a bounded domain $\Omega \subset \Rm^2$ with smooth boundary and $f,g\in C^\infty(\overline{\Omega})$, Green's first identity reads $\int_{\Omega} ((-\Delta f) g + \nabla f\cdot \nabla g) = -\int_{\partial\Omega} \partial_\nu f g$, while Green's second identity is the skew-symmetrized version $\int_{\Omega} (-g\Delta f + f\Delta g) = \int_{\partial\Omega} (f \partial_\nu g - g\partial_\nu f)$.} hold) of a one-parameter family of degenerate elliptic operators on the Euclidean unit disk $\Dm = \{z\in \Cm: |z|\le 1\}$: for $\gamma\in \Rm$, using polar coordinates $z = \rho e^{i\omega}$ and using $x = 1-\rho^2$ as the boundary defining function and with $\d V = \rho\ \d\rho\ \d\omega$ the Euclidean measure, we define 
\begin{align}
    \L_\gamma &:= -\rho^{-1} x^{-\gamma} \partial_\rho (\rho x^{\gamma+1} \partial_\rho) - \rho^{-2} \partial_\omega^2 + (1+\gamma)^2, \label{eq:Lgamma} \\
    &= - x \partial_\rho^2 - (\rho^{-1}-(3+2\gamma)\rho)\partial_\rho - \rho^{-2}\partial_\omega^2 + (\gamma+1)^2 id. \label{eq:Lgamma2}
\end{align}
Expression \eqref{eq:Lgamma} shows that $\L_\gamma$ is formally self-adjoint for the space $L^2_\gamma := L^2(\Dm, x^\gamma \d V)$ (whose inner product we denote $\dprod{f}{g}_{L^2_\gamma} := \int_{\Dm} f \bar{g}\ x^\gamma \d V$), while expression \eqref{eq:Lgamma2} shows the degenerate behavior of $\L_\gamma$, of first order in the top-degree term normal to the boundary. Although the coefficient $\gamma$ only appears in subprincipal terms, the operator-theoretic properties of $\L_\gamma$ (e.g. number of self-adjoint extensions, associated traces and their regularity at the boundary) strongly depend on $\gamma$. 

% alternative literature
There are many appearances of operators of this type in the literature:

\smallskip
$\bullet$ Such operators arise in the study of fluid flows, as operators that switch from being elliptic to hyperbolic across a curve, two prototypes of which are Keldysh (our case) and Tricomi operators, see e.g. the book \cite{Otway2012}. Indeed, expression \eqref{eq:Lgamma} of $\L_\gamma$ naturally extends past the unit circle, and the operator becomes hyperbolic outside the unit disk. Here we are interested in confining ourselves to the ``elliptic region" $\Dm$ and studying the operator $\L_\gamma$ there. 

\smallskip
$\bullet$ In recent advances on microlocal methods for asymptotically hyperbolic manifolds and spacetimes, Vasy initiated in \cite{Vasy2013,Vasy2012} a series of works which leverage some properties of Keldysh-type operators. On asymptotically hyperbolic manifolds, after compactification and change of smooth structure, Laplace-Beltrami type of operators can be factored using model Keldysh operators, and the meromorphic continuation of the resolvent of the former uses crucially that the latter enjoy {\em radial point estimates}, see also \cite{Zworski2016}. Keldysh operators also arise as the restriction of the Minkowski Laplacian on the boundary of radially-compactified Minkowski space-time. There, the elliptic regions correspond to polar caps limit points of timelike trajectories. See also \cite{Lebeau2019,Galkowski2021}. 

\smallskip
$\bullet$ Going the opposite route, one may fit Keldysh-type operators into the framework of uniformly degenerate (or 0-) operators of Mazzeo \cite{Mazzeo1991}, a framework originally designed to study (Hodge-)Laplacians on asymptotically hyperbolic manifolds, providing a flexible context where such operators can be made Fredholm. Such operators also arise as (the spatial part of) Heston diffusions in mathematical finance \cite{Feehan2015}, or Kimura diffusions in population genetics \cite{Epstein2013,Epstein2014}. In these cases, the study is further complicated by the fact that the model necessarily involves spatial domains with corners. Of study there is the diffusion associated with these operators, and regularity properties of solutions. 

\smallskip
$\bullet$ The $H^1$-type spaces constructed below are example of Sobolev spaces associated with degenerate quadratic forms in the sense of \cite{Cavalheiro2008,Sawyer2010}, though they benefit from a slightly more specific degenerate behavior of the quadratic form (only at the boundary, uniform behavior in terms of the boundary point).

\smallskip
% our motivation
The authors' motivation for this work arises in the connection of the $\L_\gamma$ family with inverse problems on Riemannian manifolds with (geodesically) convex boundary, notably, the description of appropriate Hilbert scales which capture sharp mapping properties of the geodesic X-ray transform, see e.g. \cite{Monard2017,Monard2019a, Mishra2022, Mazzeo2021, Monard2021}. Recently, the authors with Mishra proved in \cite{Mishra2022} that a one-parameter family of weighted geodesic X-ray transforms on the Euclidean disk had its normal operators be functions of (a distinguished self-adjoint realization of) $\L_\gamma$ and $i\partial_\omega$. As a step toward exploring the connections between degenerate elliptic operators and X-ray transform on more general Riemannian surfaces, this article endeavors to study the family $\L_\gamma$ in its own right, including the existence (or non-existence) of trace operators and their regularity, Dirichlet-to-Neumann map, and appropriate Hilbert scales where mapping properties are sharply described. 

% ideas 
The starting point of the work is the observation that for any $f,g\in C^\infty(\Dm^{int})$ and $R\in (0,1)$, 
\begin{align}
    \int_{\Dm_R} \left(\bg \L_\gamma f - x \partial_\rho f\ \partial_\rho \bg - \frac{1}{\rho}\partial_\omega f \frac{1}{\rho} \partial_\omega \bg - (\gamma+1)^2 f\bg\right)\ &x^\gamma \ \d V = \int_{\partial \Dm_R} \rho x^{\gamma+1} \bg \partial_\rho f \ \d\omega, \label{eq:preG1}\\
    \int_{\Dm_R} (\bg \L_\gamma f - f\L_\gamma \bg)\ x^\gamma\ \d V &= \int_{\partial \Dm_R} \rho x^{\gamma+1}(\bg \partial_\rho f - f \partial_\rho \bg)\ \d\omega, \label{eq:preG2}
\end{align}
and the classical question becomes to understand in what sense these identities can be understood as Green's first and second identities: for what spaces for $f,g$ can we send $R\to 1$ in the identities above, and make sense of the right-hand sides as boundary traces? Once this can be obtained on domains of definition of $\L_\gamma$ where it is a closed operator, the extension of identity \eqref{eq:preG2} gives a measure of how close $\L_\gamma$ is to being self-adjoint, and self-adjoint realizations of $\L_\gamma$ can be understood in terms of restrictions of $\L_\gamma$ to subspaces with specified boundary constraints. Classically, one defines the {\em minimal} operator $\L_{\gamma,min}$ to be the closure of $\L_\gamma$ equipped with domain $\dot{C}^\infty(\Dm)$ (smooth functions vanishing at infinite order at the boundary). From \eqref{eq:preG2}, the operator $\L_{\gamma,min}$ is easily seen to be symmetric, and a classical question is to understand and characterize all self-adjoint realizations of $\L_{\gamma}$ between $\L_{\gamma,min}$ and $\L_{\gamma,max} := \L_{\gamma,min}^*$. 

\medskip
% outline
We now briefly describe the main results of the article presented in the next section. 

In Section \ref{sec:mainprelim}, we first fix notation and state preliminary properties of the $\L_\gamma$ family, while recalling some distinguished self-adjoint realizations given in prior literature \cite{Wuensche2005,Mishra2022}. 

In Section \ref{sec:gammage1}, we first characterize the Friedrichs extension of $\L_\gamma$ in terms of previously known extensions, and deduce in Theorem \ref{th:minsa} that for $|\gamma|\ge 1$, the minimal operator $\L_{\gamma,min}$ is in fact self-adjoint. In particular, there is only one self-adjoint extension of $\L_{\gamma,min}$, and \eqref{eq:preG1}-\eqref{eq:preG2} can only have trivial right-hand side if extended to $R\to 1$ with $f,g$ in a domain where $\L_\gamma$ is closed. 

Section \ref{sec:gammale1} then covers the case $|\gamma|<1$, where the situation is markedly different: one may define domains where $\L_\gamma$ is closed, and where the right-hand sides of \eqref{eq:preG1}-\eqref{eq:preG2} can be extended into Dirichlet and Neumann trace operators whose precise mapping properties and \response{tangential} regularity are given in the main theorems, Theorems \ref{thm:first} and \ref{thm:second}. In this case, one can also naturally define a Dirichlet-to-Neumann map, see Theorem \ref{thm:DNmap}. The main theorems provide ways of making sense of Green's identities \eqref{eq:preG1}-\eqref{eq:preG2} when $f,g$ belong to the maximal domain (the domain of $\L_{\gamma,min}^*$), see Theorem \ref{thm:second}, or a subspace of it called $W_\gamma^2$ (see \eqref{eq:W2gamma}) in Theorem \ref{thm:first}. Unlike in Theorem \ref{th:minsa}, the operator $\L_\gamma$ now has infinitely \response{many} self-adjoint realizations, whose domains of definition are obtained by prescribing certain boundary conditions. 

To make this last point more precise, in Section \ref{sec:bt}, we reformulate our main results in the language of {\em boundary triples} and {\em generalized boundary triples}. The latter objects allow, via a general functional-analytic framework, to describe self-adjoint extensions of a given operator in terms of self-adjoint boundary relations, see e.g. \cite{Behrndt2007,Behrndt2012,Behrndt2020}, in terms of classical notions such as $\gamma$-fields (called ``Poisson maps" here to avoid conflicts with the constant $\gamma$) and Weyl $M$-functions. 

% on the methodology
We end this discussion by briefly describing the methodology. The family $\L_\gamma$ is rotation-invariant and as such gives rise to countably many one-dimensional operators $\{\L_{\gamma,n}\}_{n\in \Zm}$ on $[0,1]_\rho$ defined by the relation $\L_{\gamma} (e^{in\omega} f(\rho)) = e^{in\omega} \L_{\gamma,n} f (\rho)$. Such operators can in principle be studied using Sturm-Liouville theory. In the latter language, the endpoint $\rho=0$ is always singular, while the properties of the endpoint $\rho=1$ depend on $\gamma$ but not on $n$: the cases $\gamma\in (-1,0)$, $\gamma\in [0,1)$ and $|\gamma|\ge 1$ respectively correspond to $\rho=1$ being a ``regular point'', a ``singular point in the limit circle case'', and a ``singular point in the limit point case''. The case of $\gamma=0$ also involves a double indicial root, which requires refined analysis. We use this {\em a priori} knowledge to construct $H^1$-type function spaces (directly on $\Dm$ rather than on each separate angular Fourier mode) which are adapted to each $\L_\gamma$, some of which require a 2D version of ``quasi-derivative'', and/or log-type \response{tangential} Sobolev regularity in the case of double indicial roots. We construct a number of trace operators (as well as their right-inverses when they exist), whose boundedness is obtained by combining angular Fourier analysis with continuous families of 1D trace estimates (see e.g. Lemmas \ref{lem:1Dtrace} and \ref{lem:H1log}), or at other times make use of specific knowledge about generalized Zernike polynomials found in \cite{Wuensche2005}, see e.g. Proposition \ref{prop:tauND}. Once such trace estimates are established and their right inverses are constructed, many results follows by density, duality and functional-analytic arguments. One of the advantages of the current analysis is that it is direct and self-contained, not requiring change of smooth structure or factorization of the operator $\L_\gamma$. These results provide new families of function spaces where boundary pairings, Green's identities for $\L_\gamma$ and Fredholm settings should be naturally understood. It should also be expected that these functional settings should become robust to similar operators that may no longer be rotation-invariant.

\section{Main results} \label{sec:main}

\subsection{Preliminaries} \label{sec:mainprelim}

\paragraph{Indicial roots and conormal spaces $\A_\gamma$.}

Let $\dot{C}^\infty(\Dm)$ be the space of smooth functions on the closed unit disk $\Dm$, all of whose derivatives vanish on $\Sm^1 = \partial\Dm$, with topological dual denoted $C^{-\infty}(\Dm)$. The latter is the space of extendible distributions and will be the largest space considered in what follows. 

Let us first define natural 'smooth' spaces of definition for $\L_\gamma$. \response{Rewriting $\L_\gamma$ in the form\footnote{Note that this expression differs from that appearing in \cite[Proof of Theorem 6]{Mishra2022}, where the term $(2+2\gamma) \rho\partial_\rho$ was erroneously given as $(3+2\gamma) \rho\partial_\rho$, although this is inconsequential for the purposes of \cite[Theorem 6]{Mishra2022}.} 
\begin{align*}
    \L_\gamma = -x \Delta + (2+2\gamma) \rho\partial_\rho - \partial_\omega^2 + (\gamma+1)^2 id,
\end{align*}
where $\rho\partial_\rho$ and $\partial_\omega$ are smooth vector fields on $\Dm$ we see that $\L_\gamma$ has smooth coefficients in $\Dm$. Hence} we naturally have 
\begin{align}
    \L_\gamma \colon \dot{C}^\infty(\Dm) \to \dot{C}^\infty(\Dm).
    \label{eq:LgamCdot}
\end{align}
We now enlarge this definition to some distinguished conormal spaces. For all $\gamma \in \Rm$, $\alpha\in \Rm$ and $n\in \Zm$, a direct calculation gives
\begin{align}
    \L_\gamma (\rho^\alpha e^{in\omega}) = (\alpha+\gamma+1)^2 \rho^\alpha e^{in\omega} + (n^2-\alpha^2) \rho^{\alpha-2} e^{in\omega}. 
    \label{eq:id2}
\end{align}
Similarly,
\begin{align}
    \L_\gamma (x^\alpha) = (2\alpha+\gamma+1)^2 x^\alpha - 4\alpha (\gamma+\alpha) x^{\alpha-1}.
    \label{eq:id3}
\end{align}
In particular, the indicial roots (independent of the boundary point), are $0$ and $-\gamma$, since unless $\alpha \in \{0,-\gamma\}$, $\L_\gamma(x^\alpha)$ is more singular as $x\to 0$ than $x^\alpha$. 

In what follows, the following intertwining property, which can be checked directly, will allow us to translate what is known of $\L_\gamma$ to obtain properties on $\L_{-\gamma}$: 
\begin{align}
    \L_\gamma \circ x^{-\gamma} = x^{-\gamma} \circ \L_{-\gamma}\quad \text{ on } \quad C^\infty(\Dm^{int}), \qquad \gamma\in \Rm.
    \label{eq:interxgamma}
\end{align}
In the case $\gamma = 0$, we also have the important property
\begin{align}
    \L_0 (\log x\ f) = \log x\ \L_0 f + 4 (\rho\partial_\rho f+ f), \qquad f\in C^\infty(\Dm^{int}).
    \label{eq:id4}
\end{align}
Identity \eqref{eq:id4} can be either checked directly, or derived exploiting \eqref{eq:interxgamma} upon sending $\gamma\to 0$ in the following identity 
\begin{align*}
    \L_\gamma \circ \frac{x^{-\gamma}-1}{-\gamma} = \frac{x^{-\gamma}-1}{-\gamma} \circ \L_{-\gamma} + \frac{\L_\gamma-\L_{-\gamma}}{\gamma},
\end{align*}
using that $\frac{x^{-\gamma}-1}{-\gamma} \to \log x$ as $\gamma\to 0$.

The above discussion on indicial roots together with \eqref{eq:id4} motivates the definition of
\begin{align}
    \A_\gamma := \left\{
    \begin{array}{cc}
	x^{-\gamma} C^\infty(\Dm) + C^\infty(\Dm), & \gamma\ne 0, \\
	\log x\ C^\infty(\Dm) + C^\infty(\Dm), & \gamma = 0,
    \end{array}
    \right.
    \label{eq:Agamma}
\end{align}
\response{sub}spaces of \response{$C^\infty(\Dm^{int})$} that encode boundary behavior, each stable under $\L_\gamma$, and some of whose subspaces form core domains of self-adjointness for $\L_\gamma$. For example, for $\gamma>-1$, \cite[Theorem 6]{Mishra2022} states that $(\L_\gamma, C^\infty(\Dm))$ is essentially self-adjoint (further, its full eigendecomposition is known in terms of generalized Zernike polynomials). This result, together with the intertwining property \eqref{eq:interxgamma}, also implies immediately the existence of other self-adjoint extensions, which we state without proof: 
\begin{lemma}\label{lem:othersaext}
    For any $\gamma<1$, the operator $(\L_\gamma,x^{-\gamma}C^\infty(\Dm))$ acting on $L^2_\gamma$ is essentially self-adjoint.    
\end{lemma}

As we will see below, Dirichlet and Neumann traces, when they exist, correspond to mechanisms aiming at extracting the most and second most singular terms in the expansion of functions in $\A_\gamma$ off of $\partial\Dm$. Here, ``singular'' ordering is among polyhomogeneous terms $\{x^\zeta\log^k x\}_{\zeta\in \Cm,\ k\in \Nm_0}$, where we have 
\begin{align*}
    x^\zeta \log^k x = o (x^{\zeta'} \log^{k'} x) \quad \text{ iff } \quad \left\{
    \begin{array}{l}
	Re(\zeta)> Re(\zeta') \quad \text{or} \\
	Re(\zeta) = Re(\zeta') \quad \text{and}\quad k<k'.
    \end{array}
    \right.
\end{align*}
In this sense, Dirichlet and Neumann traces will be thought of as the coefficients in front of the terms given below
\begin{center}
    \begin{tabular}[htpb]{c||c|c|c}
	$\gamma$ & $(-\infty,0)$ & $0$ & $(0,\infty)$ \\
	\hline
	\hline
	Dirichlet term & $1$ & $\log x$ & $x^{-\gamma}$ \\
	\hline
	Neumann term & $x^{-\gamma}$ & $1$ & $1$
    \end{tabular}    
\end{center}

\paragraph{Distributional facts.} Let ${}^t\L_\gamma \colon C^{-\infty}(\Dm)\to C^{-\infty}(\Dm)$ the transpose operator of \eqref{eq:LgamCdot}. There is a natural injection $\iota_\gamma\colon L^2_\gamma\cup \A_\gamma \to C^{-\infty}(\Dm)$ given by 
\begin{align}
    \langle \iota_\gamma f, \psi\rangle := \dprod{f}{\psi}_{L^2_\gamma}, \qquad f\in L^2_\gamma\cup \A_\gamma, \qquad \psi\in \dot{C}^\infty(\Dm), 
    \label{eq:iotagam}
\end{align}
Note that $\A_\gamma\subset L_\gamma^2$ if and only if $|\gamma|<1$. We will say that a distribution $u\in C^{-\infty}(\Dm)$ 'belongs to $L^2_\gamma$ (resp. $\A_\gamma$)' if there is $f\in L^2_\gamma$ (resp. $\A_\gamma$) such that $u = \iota_\gamma f$.  

As is visible through an integration by parts, we have that $\dprod{\L_\gamma f}{\psi}_{L^2_\gamma} = \dprod{f}{\L_\gamma \psi}_{L^2_\gamma}$ for all $f\in \A_\gamma$ and $\psi\in \dot{C}^\infty(\Dm)$. This implies that 
\begin{align}
    {}^t\L_\gamma (\iota_\gamma f) = \iota_\gamma (\L_\gamma f), \qquad \forall f\in \A_\gamma. 
    \label{eq:selftranspose}
\end{align}
In particular, the 'restriction' of ${}^t \L_\gamma$ to $\dot{C}^\infty(\Dm)$ or $\A_\gamma$ through the map $\iota_\gamma$ agrees with $\L_\gamma$. Thus, for $f\in L^2_\gamma$, we'll say that $\L_\gamma f\in L^2_\gamma$ if the distribution ${}^t \L_\gamma (\iota_\gamma f)$ belongs to $L^2_\gamma$.

\paragraph{Two natural operators.} One may define two natural closed operators out of $\L_\gamma$: 

\begin{itemize}
    \item[(i)] The minimal operator $\L_{\gamma, min}$, closure of the operator defined in \eqref{eq:LgamCdot} (also called the preminimal operator), i.e. whose domain is the completion of $\dot{C}^\infty(\Dm)$ for the graph norm 
	\begin{align}
	    f\mapsto \ngamma(f) := (\|f\|^2_{L^2_\gamma} + \|\L_\gamma f\|^2_{L^2_\gamma})^{1/2}.
	    \label{eq:ngamma}
	\end{align}
    \item[(ii)] The maximal operator $\L_{\gamma,max}$, the adjoint of $\L_{\gamma,min}$, with domain 
\begin{align}
    \dom (\L_{\gamma,max}) = \left\{f\in L^2_\gamma, \quad \L_{\gamma} f \in L^2_\gamma\right\}.
    \label{eq:dommax}
\end{align}

\end{itemize}

\subsection{Characterization of Dirichlet extensions. Self-adjointness of $\L_{\gamma,min}$ for $|\gamma|\ge 1$} \label{sec:gammage1}

Let $\L_{\gamma,D}$ be the Friedrichs extension for the quadratic form $\alpha_\gamma\colon C_c^\infty(\Dm^{int})\to \Rm$ defined by 
\begin{align}
    \alpha_\gamma(f) := \dprod{\L_\gamma f}{f}_{L^2_\gamma} \stackrel{(\star)}{=} \|\sqrt{x}\ \partial_{\rho}f\|_{L^2_\gamma}^2 + \|\rho^{-1}\partial_{\omega}f\|_{L^2_\gamma}^2 + (1+\gamma)^2\|f\|_{L^2_\gamma}^2, \quad f\in C_c^\infty(\Dm^{int}),
    \label{eq:alphagamma}
\end{align}
where $(\star)$ follows from an integration by parts with no boundary term. Then we have the following characterizations: 

\begin{lemma} \label{lem:Dcharac}
    The operator $\L_{\gamma,D}$ coincides with the closure of the following essentially self-adjoint operators: 
    \begin{align*}
	(\L_{\gamma}, x^{-\gamma} C^\infty(\Dm))\quad \text{if} \quad \gamma<0 \qquad \text{ and } \qquad (\L_{\gamma}, C^\infty(\Dm)) \quad \text{if}\quad \gamma\ge 0.
    \end{align*}
\end{lemma}

The spectral decomposition of the above operators is well-known: for $\gamma\ge 0$, $\L_{\gamma,D}$ has full eigendecomposition 
\response{
\begin{align}
    \left\{ G_{n,k}^\gamma, (n+1+\gamma)^2\right\}_{n\ge 0,\ 0\le k \le n}, \quad G_{n,k}^\gamma := P_{n-k,k}^{\gamma},
    \label{eq:Gnkgamma}
\end{align}
}
where $P_{m,\ell}^\gamma$ denotes the generalized Zernike polynomials in the convention of \cite{Wuensche2005}; for $\gamma<0$, $\L_{\gamma,D}$ has full eigendecomposition $\{x^{-\gamma} \response{G_{n,k}^{-\gamma}}, (n+1-\gamma)^2\}_{n\ge 0,\ 0\le k\le n}$. 
In either case, we can define a functional calculus for $\L_{\gamma,D}$ and a Dirichlet Sobolev scale 
\begin{align}
    \wtH_D^{s,\gamma}(\Dm) := \dom (\L_{\gamma,D}^{s/2}), \quad s\in \Rm,
    \label{eq:DSob}
\end{align}
so that the following operator makes sense and is in fact an isometry
\begin{align}
    \L_{\gamma,D}^{-1} \colon \wtH^{s,\gamma}_D (\Dm) \to \wtH^{s+2,\gamma}_D (\Dm), \qquad s\in \Rm.
    \label{eq:Linv}
\end{align}

As a result of further density lemmas proved in Section \ref{sec:density}, we have the following: 
\begin{theorem} \label{th:minsa}
    If $|\gamma|\ge 1$, then $\L_{\gamma,min}$ is self-adjoint.     
\end{theorem}
In particular, such a result precludes the existence of trace maps on the maximal domain, or a Dirichlet-to-Neumann map. 

\subsection{Traces, Green's identities and Dirichlet-to-Neumann map for $|\gamma|<1$} \label{sec:gammale1}

While Theorem \ref{th:minsa} prevents the existence of more than one self-adjoint extension for $\L_{\gamma,min}$ whenever $|\gamma|\ge 1$, we now describe a markedly different scenario for any value $\gamma\in (-1,1)$. The construction has varying degrees of simplicity depending on whether $\gamma\in (-1,0)$, $\gamma=0$ or $\gamma\in (0,1)$, though in the interest of conciseness, we will unify the presentation. For each $\gamma\in (-1,1)$, there exists a radial function $\phi_\gamma$ (see \eqref{eq:phigamma}) non-vanishing on a neighborhood $[0,x_\gamma)_x \times \Sm^1_\omega$ of $\partial \Dm$, and satisfying $\L_\gamma \phi_\gamma = (\gamma+1)^2 \phi_\gamma$ on $[0,x_\gamma)_x \times \Sm^1_\omega$, with the relevant behavior 
\begin{align}
    \left\{
    \begin{array}{cl}
	\phi_\gamma \equiv 1, & \gamma\in (-1,0), \\
	\lim_{\rho\to 1} x^\gamma \phi_\gamma = 1, \qquad \lim_{\rho\to 1} x^{\gamma+1} \partial_\rho \phi_\gamma = 2\gamma, & \gamma\in (0,1), \\
	\lim_{\rho\to 1} (\phi_0/\log x) = 1, \qquad \lim_{\rho\to 1} x \partial_\rho \phi_0 = -2, & \gamma=0.	
    \end{array}
    \right.
    \label{eq:phigammaprops}
\end{align}

We may then define ``regularized" Dirichlet and Neumann traces $\dirneu\colon \A_\gamma \to C^\infty(\Sm^1)$ as 
\begin{align}
    \begin{split}
	\tau_\gamma^D f &:= (f/\phi_\gamma)|_{x=0}, \qquad \tau_\gamma^N f := W(f,\phi_\gamma)|_{x=0}, \\
	\text{where} \qquad & W(f,g) (\rho,\omega) := \rho x^{\gamma+1} (f\partial_\rho g - g \partial_\rho f).    
    \end{split}
    \label{eq:traces}
\end{align}
In particular, a direct calculation shows that, for $f\in\A_\gamma$, taking the form $f = f^{(0)} + \log x f^{(\log)}$ if $\gamma=0$, or $f = f^{(0)} + x^{-\gamma} f^{(-\gamma)}$ if $|\gamma|\in (0,1)$, with $f^{(0)}, f^{(\log)}, f^{(-\gamma)}\in C^\infty(\Dm)$, we have 
\begin{align}
    \tau_\gamma^D f = \left\{
    \begin{array}{cl}
	f^{(0)}|_{x=0}, & \gamma\in (-1,0), \\ 
	f^{(\log)}|_{x=0}, & \gamma=0, \\ 
	f^{(-\gamma)}|_{x=0}, & \gamma\in (0,1), 
    \end{array}
    \right. \quad \text{and} \quad \tau_\gamma^N f = \left\{
    \begin{array}{cl}
	-2\gamma f^{(-\gamma)}|_{x=0}, & \gamma\in (-1,0), \\ 
	-2 f^{(0)} - 2c_0 f^{(\log)}|_{x=0}, & \gamma=0, \\ 
	2\gamma f^{(0)}|_{x=0}, & \gamma\in (0,1),
    \end{array}
    \right.
    \label{eq:tracescomp}
\end{align}
where $c_0$ is the constant appearing in the definition \eqref{eq:phigamma} of $\phi_0$. In this sense, $\tau_\gamma^D$ extracts the most singular term and $\tau_\gamma^N$ extracts the second most singular term (or a linear combination of them for $\gamma=0$). Moreover, the following intertwining property follows naturally: for $\gamma\in (0,1)$ and $f\in \A_\gamma$, then $x^{\gamma}f \in \A_{-\gamma}$ and 
\begin{align}
    \tau_{\gamma}^D f = \tau_{-\gamma}^D (x^\gamma f) \quad \text{and} \quad \tau_{\gamma}^N f = \tau_{-\gamma}^N (x^\gamma f).
    \label{eq:intertrace}
\end{align}

To discuss Green's identities, we first need to define an ``$H^1$" inner product where the Dirichlet trace extends boundedly. A first guess would be to extend the form $\alpha_\gamma$ defined in \eqref{eq:alphagamma} to $\A_\gamma$, but this only makes sense for $\gamma\in (-1,0)$. \response{Indeed, for $\gamma\in (0,1)$ for instance, the last right-hand side of \eqref{eq:alphagamma} can become infinite when applied to an element of $x^{-\gamma} C^\infty(\Dm)$, and the equality ($\star$) there no longer holds; see also the discussion in Section \ref{sec:tgammab}, titled ``Case $0\le \gamma<1$''.} 

To remedy this, we let $\rho_\gamma = \sqrt{1-x_\gamma}$, and for $a, b\in [0,1]$ with $a<b$, we let $\Dm_a$ be the centered disk of radius $a$, with boundary $C_a$, and $A_{a,b}$ be the annulus $\{a<\rho<b\}$. We then define, for $f,g\in \A_\gamma$, and $b\in (\rho_\gamma, 1)$, 
\begin{align}
    \begin{split}
	\ft_{\gamma,b}[f,g] &:= \dprod{\sqrt{x}\ \phi_\gamma \partial_\rho(f/\phi_\gamma)}{\sqrt{x}\ \phi_\gamma \partial_\rho (g/\phi_\gamma)}_{x^\gamma, A_{b,1}} - b (x(b))^{\gamma+1} \frac{\partial_\rho \phi_\gamma}{\phi_\gamma}(b) \int_{C_b} f \bg   \\	
	&\qquad + \dprod{\sqrt{x}\ \partial_\rho f}{\sqrt{x}\ \partial_\rho g}_{x^\gamma, \Dm_{b}} + \dprod{\rho^{-1} \partial_\omega f}{\rho^{-1} \partial_\omega g}_{L^2_\gamma} + (\gamma+1)^2 \dprod{f}{g}_{L^2_\gamma},
    \end{split}    
    \label{eq:H1tilde}
\end{align}
where, here and below, $\int_{C_b} h$ is shorthand for $\int_{0}^{2\pi} h(b,\omega)\ \d\omega$. It can be checked for any $f\in\A_\gamma$ that $\sqrt{x}\ \partial_\rho f$ is in $L^2_\gamma$ away from the boundary, and that $\sqrt{x}\ \phi_\gamma\partial_\rho(f/\phi_\gamma)$ is in $L^2_\gamma$ near the boundary, so that \eqref{eq:H1tilde} is well-defined. 

\begin{lemma} \label{lem:tgammab}
    (1) For any $f,g\in \A_\gamma$, the definition of $\ft_{\gamma,b}$ does not depend on $b\in (b_\gamma, 1)$. We thus denote $\dprod{\cdot}{\cdot}_{\wtH^{1,\gamma}}$ the value of $\ft_{\gamma,b}$ for any $b$. 

    (2) With $\alpha_\gamma$ the form defined in \eqref{eq:alphagamma}, we have for $\gamma\in [0,1)$
	\begin{align}
	    \dprod{f}{f}_{\wtH^{1,\gamma}} = \alpha_\gamma(f), \qquad f\in \A_{\gamma,D}, \quad \text{where}\quad \A_{\gamma,D}:= \ker\tau_\gamma^D, 
	    \label{eq:AgammaD}
	\end{align}
	while for $\gamma\in (-1,0)$, the above equality holds trivially true on all of $\A_\gamma$. 

    (3) For any $f,g \in \A_\gamma$, the first Green's identity holds
    \begin{align}
	\dprod{\L_{\gamma} f}{g}_{L^2_\gamma} = \dprod{f}{g}_{\wtH^{1,\gamma}} + \dprod{\neu f}{\dir g}_{L^2(\Sm^1)}.
	\label{eq:G1}
    \end{align}    

    (4) For any $\gamma\in (-1,1)$, the form $\dprod{\cdot}{\cdot}_{\wtH^{1,\gamma}}$ defined in \eqref{eq:H1tilde} is positive definite on $\A_\gamma$. 
    \end{lemma}

Skew-symmetrizing \eqref{eq:G1}, the second Green's identity reads 
\begin{align}
    \dprod{\L_{\gamma} f}{g}_{L^2_\gamma} - \dprod{f}{\L_{\gamma} g}_{L^2_\gamma} = \dprod{\neu f}{\dir g}_{L^2(\Sm^1)} - \dprod{\neu g}{\dir f}_{L^2(\Sm^1)}, \qquad f,g\in \A_\gamma,
    \label{eq:G2}
\end{align}
which is a way of quantifying the lack of self-adjointness of $\L_\gamma$. The question is then to find spaces where $\L_\gamma$ is closed and extend the traces \eqref{eq:traces} to those spaces. 

Since by virtue of Lemma \ref{lem:tgammab}.(4), the form $\dprod{\cdot}{\cdot}_{\wtH^{1,\gamma}}$ is positive definite, we then let 
\begin{align}
    \begin{split}
	\wtH^{1,\gamma}(\Dm):& \text{ the completion of } (\A_\gamma, \dprod{\cdot}{\cdot}_{\wtH^{1,\gamma}}), \\
	\wtH^{1,\gamma}_0 (\Dm):& \text{ the completion of } (\dot{C}^\infty (\Dm), \dprod{\cdot}{\cdot}_{\wtH^{1,\gamma}}).	
    \end{split}
    \label{eq:H1tildedef}
\end{align}
We have the obvious inclusions $\wtH^{1,\gamma}_0 (\Dm)\subset \wtH^{1,\gamma}(\Dm)\subset L^2_\gamma \subset C^{-\infty}(\Dm)$. An important subspace of $\wtH^{1,\gamma}$ for what follows is 
\begin{align}
    W^2_\gamma := \{f \in \wtH^{1,\gamma}(\Dm), \quad \L_\gamma f \in L^2_\gamma \} = \dom(\L_{\gamma,max}) \cap \wtH^{1,\gamma}(\Dm),
    \label{eq:W2gamma}
\end{align}
equipped with the norm $\|f\|_{W^{2}_\gamma}^2 := \|f\|_{\wtH^{1,\gamma}}^2 + \|\L_\gamma f\|^2_{L^2_\gamma}$.

On $\Sm^1 = \partial\Dm$, we define $H_{(\gamma)}$ to be the completion of $(C^\infty(\Sm^1), \|\cdot\|_{(\gamma)})$, where for $f= \sum_{k\in \Zm} f_k e^{ik\theta}$, we define 
\begin{align}
    \|f\|_{(\gamma)}^2 := \left\{
    \begin{array}{ll}
	\sum_{k\in \Zm} \langle k\rangle^{2|\gamma|} |f_k|^2, & |\gamma|\in (0,1), \\
	\sum_{k\in \Zm} (1+\log \langle k\rangle) |f_k|^2, & \gamma=0.
    \end{array}
    \right.
    \label{eq:HgamS1}
\end{align}
For $|\gamma|\in (0,1)$, $H_{(\gamma)}$ is the classical Sobolev space $H^{|\gamma|}(\Sm^1)$ with dual identified with $H^{-|\gamma|}(\Sm^1)$. For $\gamma=0$, $H_{(0)}$ is a log-weighted Sobolev space, whose dual is identified with the completion of $C^\infty(\Sm^1)$ for the norm $f\mapsto \sum_{k\in \Zm} (1+\log \langle k\rangle)^{-1} |f_k|^2$. Below, we write $\langle \cdot, \cdot \rangle_{H_{(\gamma)}',H_{(\gamma)}}$ for the corresponding duality pairing.
Our first main theorem is the following: 
\begin{theorem}\label{thm:first}
 \response{Let $\gamma\in(-1,1)$.}  The traces $\dirneu$ defined in \eqref{eq:traces} extend as bounded operators $\dir\colon \wtH^{1,\gamma}(\Dm)\to H_{(\gamma)}$ and $\neu\colon W_\gamma^2\to H_{(\gamma)}'$, with $\dir$ onto. Green's first identity \eqref{eq:G1} extends to $f\in W_\gamma^2$ and $g\in \wtH^{1,\gamma}$
    \begin{align}
	\dprod{\L_\gamma f}{g}_{L^2_\gamma} = \dprod{f}{g}_{\wtH^{1,\gamma}} + \langle \neu f, \dir g\rangle_{H_{(\gamma)}', H_{(\gamma)}}.
	\label{eq:G1ext}
    \end{align}
    In particular, Green's second identity \eqref{eq:G2} extends to $f,g\in W_\gamma^2$
    \begin{align}
	(\L_\gamma f, g )_{L^2_\gamma} - (f, \L_\gamma g)_{L^2_\gamma} = \langle \neu f, \dir g\rangle_{H_{(\gamma)}',H_{(\gamma)}} - \langle \neu g, \dir f \rangle_{H_{(\gamma)}',H_{(\gamma)}}.
	\label{eq:G2ext}
    \end{align}  
\end{theorem}

Furthermore, for $\gamma\in (-1,1)$, one may define the analogue of a Dirichlet-to-Neumann map\response{, defined with respect to a spectral parameter $\lambda\in\rho(\L_{\gamma,D})$ (where, for any operator $B$, $\rho(B)$ denotes its resolvent set).}
\begin{theorem}\label{thm:DNmap}
    For any $\gamma\in (-1,1)$ and any $\lambda\in \rho(\L_{\gamma,D})$, there exists a bounded Dirichlet-to Neumann map operator 
    \begin{align}
	\Lambda_\gamma(\lambda)\colon H_{(\gamma)}\to H_{(\gamma)}',
	\label{eq:DNmap}
    \end{align}
    such that for any $u\in W_\gamma^2 \cap \ker (\L_\gamma-\lambda)$, $\Lambda_\gamma(\lambda) (\dir u) = \neu u$.  
\end{theorem}

Our second main result consists in extending Green's second identity \eqref{eq:G2} to the maximal domain $\dom (\L_{\gamma,max})$ defined in \eqref{eq:dommax}. This requires a splitting of the maximal domain in terms of a distinguished, well-understood self-adjoint extension of $\L_\gamma$, notably $\L_{\gamma,D}$ here. \response{Here and below, we define $\fN_\lambda (\L_{\gamma,max}) := \{f\in \dom (\L_{\gamma,max}),\ \L_\gamma f = \lambda f\}$.}

\begin{lemma} \label{lem:directsum}
    With $\dom (\L_{\gamma,D}) = \wtH^{2,\gamma}_D(\Dm)$, for any $\lambda\in \rho(\L_{\gamma,D})$, we have 
    \begin{align}
	\dom(\L_{\gamma,max}) = \dom (\L_{\gamma,D}) \oplus \fN_\lambda (\L_{\gamma,max})
	\label{eq:directsum}
    \end{align}    
\end{lemma}

For $f\in \dom (\L_{\gamma,max})$, upon fixing $\lambda\in \rho(\L_{\gamma,D})$, we write $f = f_D+f_\lambda$ the decomposition of $f$ according to \eqref{eq:directsum}. Our second main theorem is as follows. 

\begin{theorem}\label{thm:second}
    The traces $\dirneu$ defined in \eqref{eq:traces} extend as bounded, surjective operators 
    \begin{align*}
	\neu\colon \wtH^{2,\gamma}_D(\Dm)\to H^{1-|\gamma|}(\Sm^1), \qquad \tilde\tau_\gamma^D \colon \dom (\L_{\gamma,max}) \to H^{-1+|\gamma|}(\Sm^1).	
    \end{align*}
    Green's second identity \eqref{eq:G2} extends to $f,g\in \dom(\L_{\gamma,max})$ as follows:
    \begin{align}
	\dprod{\L_\gamma f}{g}_{L^2_\gamma} - \dprod{f}{\L_\gamma g}_{L^2_\gamma} = \langle \tilde\tau^D_\gamma g, \tau_\gamma^N f_D \rangle_{H^{-1+|\gamma|}, H^{1-|\gamma|}} - \langle \tilde\tau^D_\gamma f, \tau_\gamma^N g_D\rangle_{H^{-1+|\gamma|}, H^{1-|\gamma|}}. 
	\label{eq:G2extmax}
    \end{align}    
\end{theorem}

\subsection{Consequences: boundary triples, self-adjoint extensions and more} \label{sec:bt}

The theory of boundary triples is a functional-analytic framework to help describe all self-adjoint realizations of an operator and their spectrum, through the use of their Poisson maps and Weyl functions, see \cite{Behrndt2020}. Refined concepts, such as {\em generalized boundary triples} and {\em quasi boundary triples} have been introduced (see, e.g., \cite{Behrndt2007,Behrndt2012,Derkach1995}), and we now explain how and why our main Theorems \ref{thm:first} and \ref{thm:second} fit this framework, along with some natural corollaries.  

\subsubsection{Preliminaries}

\paragraph{Boundary triples and their Weyl functions.} Here and below, we follow the exposition in \cite{Behrndt2012} and refer to there for more details. Let $S$ be a densely defined, closed, symmetric operator acting on a Hilbert space $(\H,\dprod{\cdot}{\cdot}_\H)$. A triple $\{ {\cal G}, \Gamma_0, \Gamma_1\}$ is said to be a {\em boundary triple} for the adjoint operator $S^*$ if $({\cal G}, \dprod{\cdot}{\cdot}_\G)$ is a Hilbert space and $\Gamma_0,\Gamma_1\colon\dom(S^*)\to {\cal G}$ are linear mappings such that the map $(\Gamma_0, \Gamma_1)\colon \dom(S^*) \to \G\times\G$ is surjective and we have the following abstract Lagrange/Green identity, true for all $f,g\in \dom(S^*)$,
\begin{align}
    \dprod{S^* f}{g}_\H - \dprod{f}{S^*g}_\H = \dprod{\Gamma_1 f}{\Gamma_0 g}_\G - \dprod{\Gamma_0 f}{\Gamma_1 g}_\G.
    \label{eq:Lagrange}
\end{align}
By virtue of \cite[Proposition 1.2]{Behrndt2012}, a boundary triple helps parameterize all closed extensions of $S$ which are restrictions of $S^*$ in terms of closed, linear relations\footnote{\response{A closed linear relation $R$ in $\G$ is a closed linear subspace of $\G\times\G$. %Thus \[u\in\ker(\Gamma_1-\Theta\Gamma_0)\iff (\Gamma_0u,\Gamma_1u)\in\Theta,\] i.e.\ the domain of $A_\Theta$ in the above correspondence is the set of $u$ satisfying $(\Gamma_0u,\Gamma_1u)\in\Theta$.}
}} \response{$R$} in $\G$ via the bijection 
\begin{align*}
	\response{R}\mapsto A_\response{R} := S^* | \response{ (\Gamma_0,\Gamma_1)^{-1} (R).} %\ker (\Gamma_1 - \Theta \Gamma_0).
\end{align*}
\response{Given a relation $R$ on $\G$, we define the adjoint relation $R^* := \{(f,f')\in \G\times \G,\ (f',h)_\G = (f,h')_\G \text{ for all } (h,h')\in R\}$. Then the further identity} $A_{\response{R}^*} = A_{\response{R}}^*$ implies that $A_\response{R}$ is self-adjoint if and only if \response{$R$} is a self-adjoint relation in $\G$. 

\response{To say something further about the spectral properties of such extensions, we should mention that the results quoted below will sometimes constrain the relation $R$ to be of the form $\{(f,\Theta f),\ f\in \G\}$ for some linear map $\Theta\colon \G\to \G$, in which we case we will denote the extension $A_\Theta:= S^*|\ker (\Gamma_1-\Theta\Gamma_0)$, or of the form $\{(Bf, f),\ f\in \G\}$ for some linear map $B\colon \G\to \G$, in which case we will denote the extension $A_{[B]}:= S^* |\ker (B\Gamma_1-\Gamma_0)$.}

Denoting $A_0 := \response{A_{\{0\}\times \G}} = S^*|\ker\Gamma_0$, a boundary triple gives rise to a notion of {\em Poisson map}\footnote{Note that the interpretation as a Poisson map arises in PDE contexts. In abstract boundary triple theory, such a map is referred to as a $\gamma$-field, a terminology which we are avoiding here to avoid confusion with the constant $\gamma$ in $\L_\gamma$.} $P$ and a {\em Weyl function} $M$, 
\begin{align*}
    \P(\lambda) := (\Gamma_0 | \N_\lambda(S^*))^{-1}, \qquad M(\lambda) := \Gamma_1 P(\lambda), \qquad \lambda\in \rho(A_0). 
\end{align*}
By \cite[Prop. 1.5, 1.6]{Behrndt2012}, $\P$ and $M$ are holomorphic functions on $\rho(A_0)$ with values in $\B(\G,\H)$ and $\B(\G)$, respectively. 

Further, $M$ is a Herglotz function which allows to characterize the spectrum of any closed extension $A_\Theta$ of $S$ in the following way: \response{in the ``$\Theta$'' convention mentioned above,} \cite[Theorem 1.7 (ii)]{Behrndt2012} states that $\lambda\in \sigma_i(A_\Theta)$ if and only if $0\in \sigma_i (\Theta-M(\lambda))$, $i=\text{p},\text{c}, \text{r}$, where $\sigma_{\text{p}}$, $\sigma_{\text{c}}$, $\sigma_{\text{r}}$ denote the point, continuous and residual spectrum, respectively. The functions $\P$ and $M$ also allow to write Krein's formula relating the resolvent of $A_\Theta$ with that of $A_0 := S^* | \ker \Gamma_0$: for all $\lambda\in \rho(A_0) \cap \rho(A_\Theta)$, 
\begin{align}
    (A_\Theta - \lambda)^{-1} = (A_0-\lambda)^{-1} + \P(\lambda) (\Theta-M(\lambda))^{-1} \P(\overline{\lambda})^*.
    \label{eq:Krein}
\end{align}

Given the usefulness of $\P$ and $M$, the attention then turns toward their computation, or how to relate them to quantities such as a Dirichlet-to-Neumann map. As such, it is known that boundary triples for maximal operators involve some operations which make this purpose more difficult, and this motivates the definition of slightly weaker notions below, namely {\em generalized boundary triples}, first introduced in \cite{Derkach1995}, or even weaker yet, {\em quasi boundary triples}, first introduced in \cite{Behrndt2007}. 

\paragraph{Quasi/generalized-boundary triples and their Weyl functions.} A triple $\{\G,\Gamma_0,\Gamma_1\}$ is said to be a {\em quasi boundary triple} for $S^*$ if $(\G,\dprod{\cdot}{\cdot}_\G)$ is a Hilbert space and there exists an operator $T$ such that $\overline{T} = S^*$ and $\Gamma_0,\Gamma_1\colon \dom (T)\to \G$ are linear mappings such that $(\Gamma_0, \Gamma_1)\colon \dom (T)\to \G\times\G$ has dense range, \eqref{eq:Lagrange} holds for all $f,g\in \dom(T)$, and $A_0 := T|\ker \Gamma_0$ is self-adjoint in $\H$. If, in addition, $\Gamma_0$ is onto $\G$, then $\{\G,\Gamma_0,\Gamma_1\}$ is a {\em generalized boundary triple}.

Given \response{$R$} a linear relation in $\G$, we define the extension $A_\response{R}$ of $S$ in analogy to the above: 
\begin{align*}
	A_\response{R} := T| \{f\in \dom (T): (\Gamma_0 f, \Gamma_1 f) \in \response{R}\},
\end{align*}
although this no longer gives a bijective correspondence with self-adjoint extensions of $S$. Since the map $(\Gamma_0,\Gamma_1)$ is no longer surjective, we define $\G_{0} := \text{ran}(\Gamma_{0})$ and $\G_{1} := \text{ran}(\Gamma_{1})$, and we may associate to the above quasi boundary triple the Poisson map $P$ and Weyl function $M$, for any $\lambda\in \rho(A_0)$
\begin{align*}
    \P(\lambda) := (\Gamma_0|\N_\lambda(T))^{-1} \colon \G_0\to \H, \qquad M(\lambda) := \Gamma_1 \P(\lambda)\colon \G_0\to \G_1. 
\end{align*}
In the ``generalized'' boundary triple situation, $\G_0 = \G$, in which case $P(\lambda)$ and $M(\lambda)$ are defined on all of $\G$. Under additional assumptions, such functions share many properties of their ``boundary triple'' counterpart, and they are used to obtain, via \cite[Theorem 1.18]{Behrndt2012} \response{in the ``$\Theta$'' convention above}, sufficient conditions on $\Theta$ for $T|\ker (\Gamma_1-\Theta \Gamma_0)$ to be self-adjoint, allowing a parameterization of self-adjoint extensions of $S$ which are restrictions of $T$. Specifically: 

\begin{theorem}[Thm 1.18 in \cite{Behrndt2012}]\label{thm:Behrndt} 
    Let $S$ be a densely defined closed symmetric operator in $\H$ and let $\{\G,\Gamma_0,\Gamma_1\}$ be a quasi boundary triple for $\overline{T} = S^*$ with $A_i = T|\ker\Gamma_i$, $i=0,1$, and Weyl function $M$. Assume that $A_1$ is self-adjoint and that $\overline{M(\lambda_0)^{-1}}$ is a compact operator in $\G$ for some $\lambda_0\in \Cm\backslash \Rm$. If $\Theta$ is a bounded self-adjoint operator in $\G$ such that 
    \begin{align}
	\Theta(\dom \overline{M(\lambda_\pm)})\subset\G_1
	\label{eq:condBehrndt}
    \end{align}
    holds for some $\lambda_+\in \Cm^+$ and $\lambda_-\in \Cm^-$, then $A_\Theta := T|\ker(\Gamma_1-\Theta\Gamma_0)$ is a self-adjoint operator in $\H$. In particular, condition \eqref{eq:condBehrndt} is satisfied if $\text{ran}(\Theta) = \G_1$.
\end{theorem}

A more recent refinement, dropping the above compactness requirement on $\overline{M(\lambda_0)^{-1}}$, is the following result, formulated in \response{the ``$B$'' convention mentioned above}.% terms of the parameter\footnote{\response{ Note that we view $\Theta$ as a relation, and hence $\Theta^{-1}$ always makes sense as a relation; it is simply the transpose of $\Theta$ when viewing it as a subspace of $\G\times\G$. In practice we apply this to $\Theta$ such that $\Theta^{-1}$ is a bounded self-adjoint operator; this includes e.g.\ the zero operator, whose inverse in the sense of relations is the subspace $\{0\}\times\G$. }} $B = \Theta^{-1}$, in which case one denotes $A_{[B]} := T|\ker (B \Gamma_1 - \Gamma_0)$.

\begin{theorem}[Corollary 2.7 in \cite{Behrndt2017}]\label{thm:Berhndtetal}
    Let $\{\G,\Gamma_0,\Gamma_1\}$ be a quasi boundary triple for $T\subset S^*$ with corresponding Poisson map $\P$ and Weyl function $M$. Let $B$ be a bounded self-adjoint operator in $\G$ and assume that there exists a $\lambda_0\in \rho(A_0)\cap \Rm$ such that the following conditions are satisfied: 

    (i) $1\in \rho(B \overline{M(\lambda_0)})$,

    (ii) $B(\ran \overline{M(\lambda_0)})) \subset \ran \Gamma_0$,

    (iii) $B(\ran\Gamma_1) \subset \ran\Gamma_0$ or $\lambda_0\in \rho(A_1)$. 

    Then the operator $A_{[B]}\response{:= T|\ker (B \Gamma_1 - \Gamma_0)}$ is a self-adjoint extension of $S$ such that $\lambda_0 \in \rho(A_{[B]})$, and the resolvent formula 
    \begin{align}
	(A_{[B]}-\lambda)^{-1} = (A_0-\lambda)^{-1} + \P(\lambda) (I - B M(\lambda))^{-1} B \P(\bar{\lambda})^*
	\label{eq:resolvent}
    \end{align}
    holds for all $\lambda \in \rho(A_{[B]}) \cap \rho(A_0)$.     
\end{theorem}

Further, if $\{\G,\Gamma_0,\Gamma_1\}$ is a generalized boundary triple, then conditions (ii) and (iii) are automatically satisfied and $\overline{M(\lambda_0)} = M(\lambda_0)$. In that case, the conditions for self-adjointness reduce to the condition that $1\in \rho(B M(\lambda_0))$.

\subsubsection{Corollaries of Theorems \ref{thm:first} and \ref{thm:second}}

Identity \eqref{eq:G2extmax} implies that, upon defining isometric isomorphisms 
\begin{align*}
    \iota_{\pm (1-|\gamma|)} \colon H^{\pm (1-|\gamma|)}(\Sm^1) \to L^2(\Sm^1)    
\end{align*}
such that $\langle \varphi, \psi\rangle_{-1+|\gamma|,1-|\gamma|} = (\iota_- \varphi, \iota_+ \psi )_{L^2(\Sm^1)}$, we have, for all $f,g\in \dom(\L_{\gamma,max})$, 
\begin{align}
    \dprod{\L_\gamma f}{g}_{L^2_\gamma} - \dprod{f}{\L_\gamma g}_{L^2_\gamma} = \dprod{\iota_{1-|\gamma|} \tau_\gamma^N f_D}{\iota_{-1+|\gamma|} \tilde\tau_\gamma^D g}_{L^2(\Sm^1)} - \dprod{\iota_{-1+|\gamma|} \tilde\tau_\gamma^D f}{\iota_{1-|\gamma|} \tau_\gamma^N g_D}_{L^2(\Sm^1)},
    \label{eq:IBP_dommax_BT}
\end{align}
where $f_D := (\L_{\gamma,D}-\lambda)^{-1} (\L_{\gamma}-\lambda) f$ is the projection onto the left summand of \eqref{eq:directsum} for any $\lambda\in \rho(\L_{\gamma,D})$. Following the definitions above, equation \eqref{eq:IBP_dommax_BT} exactly means the following: 
\begin{corollary}\label{cor:first}
    Fix $\lambda\in \rho(\L_{\gamma,D})$ and let $f\mapsto f_D$ be the projection onto the left summand in \eqref{eq:IBP_dommax_BT}. Upon defining the maps $\Gamma_{0,1} \colon \dom(\L_{\gamma,max}) \to L^2(\Sm^1)$ by
    \begin{align}
	\Gamma_0 f := \iota_{-1+|\gamma|} \tilde\tau_\gamma^D f, \qquad \Gamma_1 g := \iota_{1-|\gamma|} \tau_\gamma^N g_D, \qquad f,g, \in \dom(\L_{\gamma,max}), 
	\label{eq:Gamma01}
    \end{align}
    where $\tilde{\tau}_\gamma^D$, $\neu$ are defined in Theorem \ref{thm:first}, the triple $\{L^2(\Sm^1), \Gamma_0, \Gamma_1\}$ is a boundary triple for the operator $\L_{\gamma,max}$.     
\end{corollary}
From the previous section, Corollary \ref{cor:first} then allows to parameterize all self-adjoint extensions of $\L_{\gamma,min}$ in terms of self-adjoint boundary relations. As mentioned in the previous section, the appearance of the map $g\mapsto g_D$ changes the computation of the $M$ function, as the latter is no longer a factorization of the Dirichlet-to-Neumann map. It could still in principle \response{be} computed, following ideas as in, e.g. \cite[Lemma 8.4.5]{Behrndt2020}, applied to Schr\"odinger operators there.

A way to simplify the description is to introduce the intermediate extension $T := \L_{\gamma,max}| W_\gamma^2$, with $W_\gamma^2$ defined in \eqref{eq:W2gamma}. Identity \eqref{eq:G2ext} implies that, upon defining isometric isomorphisms $\iota_{(\gamma)}\colon H_{(\gamma)}\to L^2(\Sm^1)$ and $\iota_{(\gamma)'} \colon H_{(\gamma)}'\to L^2(\Sm^1)$, we have for all $f,g\in W_\gamma^2$, 
\begin{align*}
    \dprod{\L_\gamma f}{g}_{L^2_\gamma} - \dprod{f}{\L_\gamma g}_{L^2_\gamma} = \dprod{\iota_{(\gamma)'} \tau_\gamma^N f}{\iota_{(\gamma)} \tau_\gamma^D g}_{L^2(\Sm^1)} - \dprod{\iota_{(\gamma)} \tau_\gamma^D f}{\iota_{(\gamma)'} \tau_\gamma^N g}_{L^2(\Sm^1)}.
\end{align*}
Moreover, the operator $\Gamma_0 := \iota_{(\gamma)} \tau_\gamma^D| W_\gamma^2$ is surjective, i.e. $\G_0 = L^2(\Sm^1)$ (indeed, inspecting the proof of Theorem \ref{thm:DNmap}, for any $f\in H_{(\gamma)}$, one constructs a function $u_f\in W_\gamma^2$ such that $\tau_\gamma^D u_f = f$) and $\Gamma_1:= \iota_{(\gamma)'} \tau_\gamma^N$ has dense range in $L^2(\Sm^1)$ since $\A_\gamma\subset W^2_\gamma$ and $\tau_\gamma^N (\A_\gamma) = C^\infty(\Sm^1)$. Together with the fact that $T|\ker\Gamma_0 = \L_{\gamma,D}$ is self-adjoint, we can then conclude: 

\begin{corollary}
    The triple $\{L^2(\Sm^1), \iota_{(\gamma)} \tau_\gamma^D , \iota_{(\gamma)'} \tau_\gamma^N \}$ is a generalized boundary triple for $S^* = \L_{\gamma,max}$ via the extension $T = \L_{\gamma,max}|W_\gamma^2$. Moreover, the Weyl function $M_\gamma(\lambda)\colon L^2(\Sm^1) \to L^2(\Sm^1)$ is given by $M_\gamma(\lambda) = \iota_{(\gamma)'} \Lambda_\gamma (\lambda) \iota_{(\gamma)}^{-1}$, where $\Lambda_\gamma$ is the Dirichlet-to-Neumann map defined in Theorem \ref{thm:DNmap}.
\end{corollary}

This result, combined with Theorem \ref{thm:Berhndtetal} and the remark right after it, allows to give a simple criterion for the self-adjointness of extensions of $\L_{\gamma,min}$ with domain included in $W_\gamma^2$, for a wide range of boundary conditions \response{including, e.g., conditions of Robin type relating both traces}. 

\begin{corollary}\label{cor:consequence} Let $\gamma\in (-1,1)$, $\{L^2(\Sm^1), \iota_{(\gamma)} \tau_\gamma^D , \iota_{(\gamma)'} \tau_\gamma^N \}$, $T$ and $M_\gamma(\lambda)$ be as in the previous result. Let $B$ be a bounded, self-adjoint operator in $L^2(\Sm^1)$ and assume that there exists $\lambda_0 \in \rho(\L_{\gamma,D})\cap \Rm$ such that 
    \begin{align*}
	1\in \rho(B M_\gamma(\lambda_0)).
    \end{align*}
    Then the operator $\L_{\gamma,[B]} := T|\ker (B \iota_{(\gamma)'} \tau_\gamma^N - \iota_{(\gamma)} \tau_\gamma^D)$ is a self-adjoint extension of $\L_{\gamma,min}$ such that $\lambda_0 \in \rho(\L_{\gamma,[B]})$ and the resolvent formula 
    \begin{align*}
	(\L_{\gamma,[B]}-\lambda)^{-1} = (\L_{\gamma,D}-\lambda)^{-1} + \P_\gamma(\lambda) (I - BM_\gamma(\lambda))^{-1} B \P_\gamma (\bar{\lambda})^*
    \end{align*}
    holds for all $\lambda\in \rho(\L_{\gamma,[B]}) \cap \rho(\L_{\gamma,D})$, where $P_\gamma(\lambda) := (\iota_{(\gamma)} \tau_\gamma^D|\N_\lambda (\L_{\gamma,max}))^{-1}$.
\end{corollary}

Refinements of the above results \response{may leverage norm estimates of $M_\gamma(\lambda_0)$ to show that the remaining condition is always satisfied, see e.g. \cite{Behrndt2017} situations where $\|M(\lambda)\|\to 0$ as $\lambda\to -\infty$. A more in-depth inquiry of such aspects, particularizing to certain boundary conditions, will be the object of future work.}%, particularizing to certain boundary conditions, will be the object of future work. 

\paragraph{Outline of next sections.} The remainder of the article is organized as follows. In Section \ref{sec:gammabig}, we treat the case $|\gamma|\ge 1$, first formulating some density results in Sec. \ref{sec:density}, then showing the characterization of the Friedrichs extension (Lemma \ref{lem:Dcharac}) in Sec. \ref{sec:dissa}, finally showing that $\L_{\gamma,min}$ is self-adjoint (Theorem \ref{th:minsa}) in Sec. \ref{sec:minsa}. In section \ref{sec:prooffirst}, we cover the construction of a generalized boundary triple (Theorem \ref{thm:first}) and a Dirichlet-to-Neumann map (Theorem \ref{thm:DNmap}) for $\L_{\gamma,max}$ when $\gamma\in (-1,1)$, with a refined roadmap of proofs at the beginning of the section. In section \ref{sec:proofsecond}, we cover the construction of a(n ordinary) boundary triple (Theorem \ref{thm:second}) for $\L_{\gamma,max}$ when $\gamma\in (-1,1)$, with a refined roadmap of proofs at the beginning of the section. Some proofs of auxiliary lemmas are provided in Appendix \ref{eq:auxlemmas}. 

%%%%%%%%%%%%%%%%%%%%%%%%%%%%%%%%%%%%%%%%%%%%%%%%%%%%%%%%%%%%%%%%%%%%%%%%%%%%%%%%%%%%%%%%%%%%%%%%%%%%%%%%%%%%%%%%%%%%%%%%%%%%%%%%%%%%%%%%%%%%%%%%%%%%%%%%%%%%
%%%%%%%%%%%%%%%%%%%%%%%%%%%%%%%%%%%%%%%%%%%%%%%%%%%%%%%%%%%%%%%%%%%%%%%%%%%%%%%%%%%%%%%%%%%%%%%%%%%%%%%%%%%%%%%%%%%%%%%%%%%%%%%%%%%%%%%%%%%%%%%%%%%%%%%%%%%%
%%%%%%%%%%%%%%%%%%%%%%%%%%%%%%%%%%%%%%%%%%%%%%%%%%%%%%%%%%%%%%%%%%%%%%%%%%%%%%%%%%%%%%%%%%%%%%%%%%%%%%%%%%%%%%%%%%%%%%%%%%%%%%%%%%%%%%%%%%%%%%%%%%%%%%%%%%%%
%%%%%%%%%%%%%%%%%%%%%%%%%%%%%%%%%%%%%%%%%%%%%%%%%%%%%%%%%%%%%%%%%%%%%%%%%%%%%%%%%%%%%%%%%%%%%%%%%%%%%%%%%%%%%%%%%%%%%%%%%%%%%%%%%%%%%%%%%%%%%%%%%%%%%%%%%%%%

\section{Density results, proofs of Lemma \ref{lem:Dcharac} and Theorem \ref{th:minsa}} \label{sec:gammabig}

\subsection{Some preliminary density results}\label{sec:density}

We begin by exploring when $C_c^{\infty}(\Dm)$ is dense in spaces of the form $(\log x)^k x^{\alpha}C^{\infty}(\Dm)$, with respect to the $L^2_\gamma$ topology and the graph-norm topology $n_\gamma$ defined in \eqref{eq:ngamma}. 

\begin{theorem}\label{thm:densities}
    (i) $C_c^{\infty}(\Dm)$ is dense in $L^2_\gamma$ for all $\gamma\in\mathbb{R}$, and in particular it is dense in $(\log x)^k x^{\alpha}C^\infty(\Dm)$ with respect to the $L^2_\gamma$ topology for all $\alpha>(-\gamma-1)/2$ and $k\in\Nm_0$.

    (ii) $C_c^{\infty}(\Dm)$ is dense in $(\log x)^k x^{\alpha}C^\infty(\Dm)$ with respect to the graph norm $n_\gamma$ for all $\gamma\in \Rm$, $\alpha>(1-\gamma)/2$ and $k\in\Nm_0$.
\end{theorem}

To prove the density results, we use the following ``cutoff'' lemma:

\begin{lemma} \label{lem:cutofflemma}
    Let $\chi\in C_c^{\infty}(\Rm)$ and $g\in (\log x)^k x^{\alpha}C^{\infty}(\Dm)$ for $k\in\Nm_0$ and $\alpha\in\Rm$. Then, if $2\alpha+\gamma>-1$, or if $\chi$ vanishes to infinite order at $0$, we have
    \[\|\chi(x/\epsilon) g\|_{L^2_{\gamma}} = O(|\log\epsilon|^k\epsilon^{\alpha+(\gamma+1)/2}), \qquad \text{as } \epsilon\to 0^+.\]
\end{lemma}

\begin{proof}[Proof of Lemma \ref{lem:cutofflemma}]
    Denote $C_g := \max_{\Dm}|(\log x)^{-k} x^{-\alpha}g|$, finite by assumption. Then we have
    \begin{align*}
	\|\chi(x/\epsilon)g\|_{L^2_{\gamma}}^2 &= \int_{\Dm}{|\chi(x/\epsilon)|^2|g|^2\,x^{\gamma}\, \d V} \\
	&\le C_g^2 \int_{\Sm^1}\int_0^1{|\chi(x/\epsilon)|^2 x^{2\alpha+\gamma}|\log x|^{2k}\ \d x\ \d \omega} \\
	&\!\!\stackrel{x=\varepsilon y}{=} 2\pi C_g^2 \int_0^{1/\epsilon}{|\chi(y)|^2(\epsilon y)^{2\alpha+\gamma}|\log(\epsilon y)|^{2k}\,\epsilon\ \d y} \\
	&\le \sum_{i=0}^{2k}{C_i|\log\epsilon|^i\epsilon^{2\alpha+\gamma+1}},\qquad C_i = 2\pi\binom{2k}{i} C_g \int_0^{\infty}{|\chi(y)|^2y^{2\alpha+\gamma}|\log(y)|^{2k-i}\ \d y},
    \end{align*}
    where all constants are finite under the assumption that $2\alpha+\gamma>-1$ or $\chi$ vanishes at infinite order.
\end{proof}

\begin{proof}[Proof of Theorem \ref{thm:densities}]
    Proof of (i). Since the weight $x^\gamma$ does not vanish on a set of positive measure on $\Dm$, it follows that functions in $L^2_\gamma$ can be approximated by functions whose support is a compact subset of the \emph{interior} of $\Dm$. Such functions can be approximated by $C_c^\infty(\Dm)$ functions in the standard way, noting that $x^\gamma$ is strictly positive on any compact subset of the interior of $\Dm$. The second part follows since $x^{\alpha}C^\infty(\Dm)\subset L^2_\gamma$ for all $\alpha>(-\gamma-1)/2$.
    
    To prove (ii), we argue by cutting off. That is, fix $\chi\in C_c^\infty(\Rm)$ with $\chi\equiv 1$ in a neighborhood of $0$, and for $f\in x^{\alpha}C^\infty(\Dm)$, let
    \[f_{\epsilon} = (1-\chi(x/\epsilon))f.\]
    Then $f_{\epsilon}\in C_c^{\infty}(\Dm)$ for $\epsilon>0$. We have $f_{\epsilon}\to f$ in $L^2_{\gamma}$, so it suffices to check if $\L_\gamma f_{\epsilon}\to \L_\gamma f$ in $L^2_\gamma$. We have
    \[\L_\gamma f_{\epsilon} = -\rho^{-1} x^{-\gamma} \partial_\rho (\rho\ x^{\gamma+1} \partial_\rho f_{\epsilon}) - \rho^{-2} \partial_\omega^2f_{\epsilon} + (1+\gamma)^2f_{\epsilon}.\]
    The third term converges to $(1+\gamma)^2f$ in $L^2_\gamma$, while the second term equals $\rho^{-2}(1-\chi(x/\epsilon))\partial_\omega^2f$, which converges to to $\rho^{-2}\partial_\omega^2f$ in $L^2_\gamma$. It remains to check that $\|\rho^{-1} x^{-\gamma} \partial_\rho (\rho\ x^{\gamma+1} \partial_\rho [f-f_{\epsilon}])\|_{L^2_\gamma} \to 0$ as $\epsilon\to 0$. We thus write
    \begin{align}
	-\rho^{-1} x^{-\gamma} \partial_\rho (\rho x^{\gamma+1} \partial_\rho [f-f_{\epsilon}])&= -\rho^{-1}x^{-\gamma}\partial_\rho (\rho\ x^{\gamma+1}\partial_\rho(\chi(x/\epsilon)f)) \nonumber\\
	&=-\rho^{-1}x^{-\gamma}\partial_\rho (\rho\ x^{\gamma+1}\chi(x/\epsilon)\partial_{\rho} f - 2\rho^2\ x^{\gamma+1}\epsilon^{-1}\chi'(x/\epsilon) f) \nonumber\\
	&= \chi(x/\epsilon)g_0 + \epsilon^{-1}\chi'(x/\epsilon)g_1 + \epsilon^{-2}\chi''(x/\epsilon)g_2 \label{eq:sumg012}
    \end{align}
    where
    \begin{align}
	\begin{split}
	    g_0 &= -\rho^{-1}\ x^{-\gamma}\partial_\rho(\rho\ x^{\gamma+1}\partial_{\rho}f), \\
	    g_1 &= -\rho^{-1}\ x^{-\gamma}(-2\rho^2\ x^{\gamma+1})\partial_{\rho}f - \rho^{-1}\ x^{-\gamma}\partial_{\rho}(-2\rho^2\ x^{\gamma+1}f) \\
	    &= 4\rho\ x\partial_{\rho}f + (4x-4(\gamma+1)\rho^2)f, \\
	    g_2 &= -4\rho^2 x\ f.	    
	\end{split}
	\label{eq:g012}	
    \end{align}
    We now note that if $f\in (\log x)^k x^{\alpha}C^\infty(\Dm)$, then the functions $g_0, g_1, g_2$ defined in \eqref{eq:g012} satisfy $g_i\in \sum_{j=\max(0,k-i)}^k{(\log x)^j x^{\alpha-1+i}C^\infty(\Dm)}$, and hence by Lemma \ref{lem:cutofflemma} we have
    \[\|\epsilon^{-i}\chi^{(i)}(x/\epsilon)g_i\|_{L^2_\gamma} = O(|\log\epsilon|^k\epsilon^{-i+(\alpha-1+i)+(\gamma+1)/2}) = O(|\log\epsilon|^k\epsilon^{\alpha-(1-\gamma)/2})\]
    for $i=0,1,2$. It follows that if $\alpha>(1-\gamma)/2$, combining the above with \eqref{eq:sumg012}, we arrive at
    \begin{align*}
	-\rho^{-1}\ x^{-\gamma} \partial_\rho (\rho\ x^{\gamma+1} \partial_\rho [f-f_{\epsilon}]) = O_{L^2_\gamma}(|\log\epsilon|^k\epsilon^{\alpha-(1-\gamma)/2}) = o_{L^2_\gamma}(1), \qquad \text{as } \epsilon\to 0.	
    \end{align*}
\end{proof}

\subsection{Friedrichs extensions and proof of Lemma \ref{lem:Dcharac}}\label{sec:dissa}
 
Recall the following result which holds for any symmetric and positive (more generally, semibounded) operator (see, e.g., \cite[Th 4.4 p34]{Helffer2013}).
\begin{theorem}\label{th:Friedrichs}
    A symmetric, positive operator $T_0$ on ${\cal H}$ (with $D(T_0)$ dense in ${\cal H}$) admits a self-adjoint extension, called the Friedrichs extension. 
\end{theorem}

The construction goes as follows.  To $T_0$ we associate the form $\alpha(u,v) := ( T_0u,v)_{\mathcal{H}}$, with domain $D(T_0)$. We then consider the Hilbert space completion $V$ of $D(T_0)$ with respect to the norm $u\mapsto(\alpha(u,u))^{1/2}$, viewing $V$ as injecting into $\mathcal{H}$, and we extend $\alpha$ by continuity to a form $\tilde\alpha$ with domain $V$ (the form will be continuous with respect to the norm on $V$). Finally let
\begin{align}
    D_F = \{u\in V\,:\, \exists\,C>0\ \text{ s.t. }|\tilde\alpha(u,v)|\le C\|v\|_{\mathcal{H}}\},
    \label{eq:DF}
\end{align}
i.e. $D_F$ consists of $u\in V$ where $\tilde\alpha(u,\cdot)$ is a bounded anti-linear functional on $\mathcal{H}$. For such $u$, there exists a unique element $T_Fu\in\mathcal{H}$ such that $\tilde\alpha(u,\cdot) = ( T_Fu,\cdot)_{\mathcal{H}}$; then $T_F$ with domain $D_F$ is the Friedrichs extension. Note that $D(T_0)\subset D_F\subset D(T_0^*)$, with $T_0^*|_{D_F} = T_F$ and $T_F|_{D(T_0)} = T_0$. 

We now give a criterion for when a self-adjoint extension is the Friedrichs extension, whose proof is relegated to Appendix \ref{app:Friedrichs}. Recall that given a symmetric form $\alpha$, we say that a sequence $u_n$ in the domain of $\alpha$ will \emph{$\alpha$-converge} to $u\in\mathcal{H}$ (not necessarily in the domain of $\alpha$) if
\[u_n\xrightarrow{n\to\infty} u\text{ in }\H\quad\text{and}\quad\alpha(u_n-u_m,u_n-u_m)\xrightarrow{m,n\to\infty} 0.\]
\begin{lemma}[Characterization of the Friedrichs extension]
\label{lem:friedrichs-lemma}
Let $T_0$ be a symmetric semibounded operator, and let $T$ be a self-adjoint extension of $T_0$. Then $T$ is the Friedrichs extension of $T_0$ if and only if $D(T)$ is contained in the $\alpha$-closure of $D(T_0)$, that is,
\begin{align*}
    \text{for all }u\in D(T)\text{, there exists a sequence }u_n\in D(T_0)\text{ such that }u_n\to_{\alpha}u.
\end{align*}
\end{lemma}

\begin{proof}[Proof of Lemma \ref{lem:Dcharac}]
    We aim to apply Lemma \ref{lem:friedrichs-lemma}, working with the quadratic form $\alpha_\gamma$ defined in \eqref{eq:alphagamma}. The proof amounts to showing that the $\alpha_\gamma$-closure of $C_c^\infty(\Dm)$ contains (i) $C^\infty(\Dm)$ if $\gamma\ge 0$ and (ii) $x^{-\gamma}C^\infty(\Dm)$ if $\gamma<0$. The domain of the $\alpha_\gamma$-closure is also closed with respect to the graph norm, hence it contains the graph norm closure of $C^\infty(\Dm)$ (resp. $x^{-\gamma}C^\infty(\Dm)$) for $\gamma\ge 0$ (resp. $\gamma<0$).

We briefly explain why (i) implies (ii). For $\gamma<0$ and $f\in x^{-\gamma}C^\infty(\Dm)$, assuming (i) is true, $x^\gamma f$ can be approximated in the $\alpha_{-\gamma}$ norm by a sequence $f_n$ in $C_c^\infty(\Dm)$. Via the intertwining property \eqref{eq:interxgamma}, we then see that $\alpha_{-\gamma}(f_n- x^\gamma f) = \alpha_\gamma (x^{-\gamma} f_n- f)\to 0$, where $x^{-\gamma} f_n\in C_c^\infty(\Dm)$. Thus $x^{-\gamma}C^\infty(\Dm)$ is contained in the $\alpha_\gamma$-closure of $C_c^\infty(\Dm)$, and (ii) holds. 

On to the proof of (i), we recall, for $\gamma\ge 0$, that the closure of $(\L_\gamma, C^\infty(\Dm))$ is a self-adjoint operator with eigenbasis \response{ $\{G^\gamma_{n,k}\}_{n\ge 0,\ 0\le k\le n}$ as defined in \eqref{eq:Gnkgamma},} satisfying
\[\L_\gamma G^\gamma_{n,k} = (n+1+\gamma)^2 G^\gamma_{n,k}, \qquad G_{n,k}^\gamma|_{\Sm^1} (e^{i\omega}) = e^{i(n-2k)\omega}, \]
\response{where the second property uses \cite[Eq. (2.10)]{Wuensche2005}.} The domain equals $\wtH_D^{2,\gamma}$ as defined in \eqref{eq:DSob}, also characterized as 
\begin{align*}
    \wtH_D^{s,\gamma} = \left\{f = \sum_{n,k}{f_{n,k}\ \widehat{G_{n,k}^{\gamma}}}\,:\,\sum_{n,k}{(n+\gamma+1)^{2s}|f_{n,k}|^2}<\infty\right\}    
\end{align*}
where $\widehat{G_{n,k}^{\gamma}} = G_{n,k}^{\gamma}/ \|G_{n,k}^{\gamma}\|_{L^2_\gamma}$. In light of Lemma \ref{lem:friedrichs-lemma}, it suffices to show: for each $f\in \wtH_D^{2,\gamma}$, there exists $f^{(j)}$ in the $\alpha_\gamma$-closure of $C_c^\infty(\Dm)$, such that
\[f^{(j)}\xrightarrow{j\to\infty}_{\alpha_\gamma} f.\]
Noting that every $f$ can be approximated in $\wtH_D^{2,\gamma}$ by a finite linear combination of the Zernike polynomials $G_{n,k}^{\gamma}$, and that a sequence convergent in $\wtH_D^{2,\gamma}$ will also $\alpha_\gamma$-converge, it suffices to show the above approximation when $f = G_{n,k}^{\gamma}$.

We note, in light of Theorem \ref{thm:densities}.(ii), that $xC^\infty(\Dm) \subset \overline{C_c^\infty(\Dm)}^{n_\gamma}$ for all $\gamma>-1$ (in particular for $\gamma\ge 0$). In particular, $xC^\infty(\Dm)$ is contained in the $\alpha_\gamma$-closure of $C_c^\infty(\Dm)$. Thus, we now aim to show: if $f = G_{n,k}^\gamma$ for some $n,k$, then there exists a sequence $f^{(j)}\in xC^\infty(\Dm)$ such that
\[f^{(j)}\xrightarrow{j\to\infty}_{\alpha_\gamma} f.\]
Note that $f\in C^\infty(\Dm)$ belongs to $xC^\infty(\Dm)$ if and only if $f|_{\Sm^1} = 0$. Furthermore, for a polynomial written in the form $\sum_{n,k}{f_{n,k}G_{n,k}^{\gamma}}$, we have
\begin{align*}
    \sum_{n,k}{f_{n,k}G_{n,k}^{\gamma}}\big|_{\Sm^1} (e^{i\omega}) = \sum_{m\in\mathbb{Z}}{\left(\sum_{n-2k=m}{f_{n,k}}\right)e^{im\omega}}.
\end{align*}
As such, for a fixed choice of $(n,k)$, let $m=n-2k$, and choose an ansatz
\begin{align}
    f^{(j)} = G_{n,k}^{\gamma} - \sum_{n'-2k'=m}{f_{n',k'}^{(j)}G_{n',k'}^{\gamma}}
    \label{eq:ansatz}
\end{align}
where, for each $j$, only finitely many coefficients $f_{n',k'}^{(j)}$ are nonzero so that $f^{(j)}\in C^{\infty}(\Dm)$, and we have
\begin{align}
    \sum_{n'-2k'=m}{f_{n',k'}^{(j)}} = 1
    \label{eq:ansatz2}
\end{align}
so that $\tau_{\gamma}^Nf^{(j)} = 0$, i.e. $f^{(j)}\in xC^\infty(\Dm)$. We aim to choose the coefficients so that
\[\|f^{(j)}-G_{n,k}^{\gamma}\|_{\wtH_D^{1,\gamma}}\xrightarrow{j\to\infty} 0.\]
Note that
\begin{align*} 
    \|f^{(j)}-G_{n,k}^{\gamma}\|_{\wtH_D^{1,\gamma}}= \left\|\sum_{n'-2k'=m}{f_{n',k'}^{(j)}G_{n',k'}^{\gamma}}\right\|_{\wtH_D^{1,\gamma}}^2 = \sum_{n'-2k'=m}{(n'+1+\gamma)^2\|G_{n',k'}^{\gamma}\|_{L^2_\gamma}^2|f_{n',k'}^{(j)}|^2}.
\end{align*}
We now use the following elementary \response{lemma, whose proof appears in Appendix \ref{sec:RAlemmas}:
\begin{lemma}\label{real-analysis-lemma}
    Suppose $\{a_k\}_{k=1}^\infty$ is a sequence of positive numbers satisfying $\sum_{k=1}^\infty{1/a_k}=\infty$. Then there exists a sequence of sequences $\{\{c_k^{(j)}\}_{k=1}^\infty\}_{j=1}^\infty$, with only finitely many numbers $c_k^{(j)}$ nonzero for any fixed $j$, satisfying
    \[\sum_{k=1}^\infty{c_k^{(j)}} = 1, \qquad\text{and}\quad \lim_{j\to\infty}\sum_{k=1}^\infty{a_k|c_k^{(j)}|^2} = 0.\]
\end{lemma}}
Thus, we aim to find $f_{n',k'}^{(j)}$ satisfying \eqref{eq:ansatz2} and such that 
\[\sum_{n'-2k'=m}{(n'+1+\gamma)^2\|G_{n',k'}^{\gamma}\|_{L^2_\gamma}^2|f_{n',k'}^{(j)}|^2}\xrightarrow{j\to\infty} 0.\]
By \response{Lemma \ref{real-analysis-lemma}}, applied to $a_{k'} = (n'+1+\gamma)^2\|G_{n',k'}^{\gamma}\|_{L^2_\gamma}^2$ and $c_{k'}^{(j)} = f_{n',k'}^{(j)}$ (where $n'-2k'=m$), we see that this is possible, provided that
\[\sum_{n'-2k'=m}{\frac{1}{(n'+1+\gamma)^2\|G_{n',k'}^{\gamma}\|_{L^2_\gamma}^2}} = \infty.\]
Note that, for fixed $m$, we can parametrize the terms in the above sum by $k'$, with $k'\in\mathbb{N}$, $k'\ge\max(0,-m)$, with
\[n' = 2k'+m\implies n'-k' = k'+m.\]
Thus
\begin{align*}
    (n'+1+\gamma)^2\|G_{n',k'}^{\gamma}\|_{L^2_\gamma}^2 &= (n'+1+\gamma)^2\frac{\pi (n'-k')!(k')!(\gamma!)^2}{(n'+\gamma+1)(n'-k'+\gamma)!(k'+\gamma)!} \\
    &=(2k'+m+1+\gamma)^2\frac{\pi (k'+m)!(k')!(\gamma!)^2}{(2k'+m+\gamma+1)(k'+m+\gamma)!(k'+\gamma)!},
\end{align*}
and the latter is asymptotic to $2\pi(\gamma!)^2(k')^{1-2\gamma}$ as $k'\to \infty$. It follows that
\[\sum_{n'-2k'=m}{\frac{1}{(n'+1+\gamma)^2\|G_{n',k'}^{\gamma}\|_{L^2_\gamma}^2}} = \infty\]
as long as $1-2\gamma\le 1$, or equivalently, $\gamma\ge 0$.
\end{proof}

Inspecting the proof above, we can in fact prove the following result, which will be useful in later sections.  
\begin{corollary}\label{cor:H1dense}
    For $\gamma\in (-1,1)$, $C_c^\infty(\Dm^{int})$ is dense in $\A_{\gamma,D}$ for the $\wtH^{1,\gamma}$ topology. 
\end{corollary}

\begin{proof} Recall that the quadratic forms $\alpha_\gamma$ and $\dprod{\cdot}{\cdot}_{\wtH^{1,\gamma}}$ agree on $C_c^\infty(\Dm)$. For $\gamma\in (-1,0)$, $\A_{\gamma,D} = xC^\infty(\Dm) + x^{-\gamma} C^\infty(\Dm)$, where $x^{-\gamma} C^\infty(\Dm)$ is included in $\overline{C_c^\infty(\Dm^{int})}^{\alpha_\gamma}$ by the previous proof, and $xC^\infty(\Dm)$ is included in $\overline{C_c^\infty(\Dm^{int})}^{n_\gamma}$ by Theorem \ref{thm:densities}.(ii), which itself is included in $\overline{C_c^\infty(\Dm^{int})}^{\alpha_\gamma}$ (this is because, by Cauchy-Schwarz inequality, we immediately have $\alpha_\gamma(f) \le n_\gamma(f)^2$). The proofs for the cases $\gamma=0$ and $\gamma\in (0,1)$ are similar. 
\end{proof}

\subsection{Proof of Theorem \ref{th:minsa}} \label{sec:minsa}

We are now ready to prove Theorem \ref{th:minsa}, namely that $\L_{\gamma,min}$ is self-adjoint for $|\gamma|\ge 1$, by showing that $\L_{\gamma,min} = \L_{\gamma,D}$. This amounts to proving that (i) for $\gamma\ge 1$, $\overline{C_c^\infty(\Dm)}^{n_\gamma} = \overline{C^\infty(\Dm)}^{n_\gamma}$, and (ii) for $\gamma \le -1$, $\overline{C_c^\infty(\Dm)}^{n_\gamma} = \overline{x^{-\gamma}C^\infty(\Dm)}^{n_\gamma}$. 

Assuming (i) holds, (ii) can be proved as follows. Fix $\gamma\le -1$ and $f\in x^{-\gamma}C^\infty(\Dm)$. By (i), $x^{\gamma}f \in C^\infty(\Dm)$ can be approximated in $n_{-\gamma}$ by a sequence $f_n\in C_c^\infty(\Dm^{int})$, and then, using the intertwining \eqref{eq:interxgamma}, $\|x^{-\gamma}f_n - f\|_{n_\gamma} = \|f_n-x^\gamma f\|_{n_{-\gamma}}\to 0$ as $n\to \infty$, i.e. the sequence $x^{-\gamma}f_n \in C_c^\infty(\Dm^{int})$ approximates $f$ in the $n_\gamma$ topology. 

We now prove (i), fixing $\gamma\ge 1$. If $\gamma>1$, then the conclusion follows from Theorem \ref{thm:densities}.(ii). The limiting case $\gamma=1$ requires special care, and the proof that follows works for all $\gamma\ge 1$ {\it a fortiori}. By Theorem \ref{thm:densities}.(ii), we have that $x C^\infty(\Dm) \subset \overline{C_c^\infty(\Dm^{int})}^{n_\gamma}$, and hence it remains to show that $C^\infty(\Dm) \subset \overline{xC^\infty(\Dm)}^{n_\gamma}$. Since each $f\in C^\infty(\Dm)$ can be written as an $n_\gamma$-convergent sum of Zernike polynomials, it suffices to show that every Zernike polynomial $G_{n,k}^\gamma$ can be $n_\gamma$-approximated by elements of $xC^\infty(\Dm)$. Thus fixing $(n,k)$ and setting $m=n-2k$, we aim to approximate $G_{n,k}^\gamma$ with a sequence $\{f^{(j)}\}_j$ in $xC^\infty (\Dm)$ of the form \eqref{eq:ansatz}, whose coefficients are chosen to satisfy \eqref{eq:ansatz2} to enforce $f^{(j)}|_{\Sm^1} =0$, and so that 
\[\|f^{(j)}-G_{n,k}^{\gamma}\|_{\wtH_D^{2,\gamma}} = \left\|\sum_{n'-2k'=m}{f_{n',k'}^{(j)}G_{n',k'}^{\gamma}}\right\|_{\wtH_D^{2,\gamma}}\xrightarrow{j\to\infty} 0.\]
In this case, we have
\[
    \left\|\sum_{n'-2k'=m}{f_{n',k'}^{(j)}G_{n',k'}^{\gamma}}\right\|_{\wtH_D^{2,\gamma}}^2 =\sum_{n'-2k'=m}{(n'+1+\gamma)^4\|G_{n',k'}^{\gamma}\|_{L^2_\gamma}^2|f_{n',k'}^{(j)}|^2}.
\]
By \response{Lemma \ref{real-analysis-lemma}}, applied to $a_{k'} = (n'+1+\gamma)^4\|G_{n',k'}^{\gamma}\|_{L^2_\gamma}^2$ and $c_{k'}^{(j)} = f_{n',k'}^{(j)}$ (where $n'-2k'=m$), we can find our desired sequences, provided that
\[\sum_{n'-2k'=m}{\frac{1}{(n'+1+\gamma)^4\|G_{n',k'}^{\gamma}\|_{L^2_\gamma}^2}} = \infty.\]
Note that, for fixed $m$, we can parametrize the terms in the above sum by $k'$, with $k'\in\mathbb{N}$, $k'\ge\max(0,-m)$, with
\[n' = 2k'+m\implies n'-k' = k'+m.\]
Thus
\begin{align*}
    (n'+1+\gamma)^4\|G_{n',k'}^{\gamma}\|_{L^2_\gamma}^2 &= (n'+1+\gamma)^4\frac{\pi (n'-k')!(k')!(\gamma!)^2}{(n'+\gamma+1)(n'-k'+\gamma)!(k'+\gamma)!} \\
    &=(2k'+m+1+\gamma)^4\frac{\pi (k'+m)!(k')!(\gamma!)^2}{(2k'+m+\gamma+1)(k'+m+\gamma)!(k'+\gamma)!},
\end{align*}
which is asymptotic to $2\pi(\gamma!)^2(k')^{3-2\gamma}$ as $k'\to \infty$. It follows that
\[\sum_{n'-2k'=m}{\frac{1}{(n'+1+\gamma)^4\|G_{n',k'}^{\gamma}\|_{L^2_\gamma}^2}} = \infty\]
as long as $3-2\gamma\le 1$, i.e. $\gamma\ge 1$. Thus, we can construct $f^{(j)}\in xC^\infty(\Dm)\subset\text{dom }(\L_{\gamma,min})$ approximating any $G_{n,k}^\gamma$ in the graph norm, and this concludes the proof of the case $\gamma\ge 1$. The proof of Theorem \ref{th:minsa} is complete. 

\section{Proof of Lemma \ref{lem:tgammab}, Theorem \ref{thm:first} and \ref{thm:DNmap}} \label{sec:prooffirst}

This section aims at proving the first main theorem, Theorem \ref{thm:first}. This is done through the following roadmap. 

In Section \ref{sec:tgammab}, we first prove Lemma \ref{lem:tgammab}, which consists in making sense of the appropriate $\wtH^{1,\gamma}$ inner product for which Green's identities hold for functions in $\A_\gamma$. Section \ref{sec:Dext} then covers the surjective extension of $\tau_\gamma^D$ to $\wtH^{1,\gamma}$, as stated in Theorem \ref{thm:qbt1} below. This requires elementary one-dimensional traces estimates as derived in Section \ref{sec:1Dtraces}, combined with Fourier expansions in the boundary coordinate to complete the proof of Theorem \ref{thm:qbt1} in Section \ref{sec:qbt1proof}. In section \ref{sec:prooffirstthm}, we first prove preliminary properties in Lemma \ref{lem:H1spaces}: characterizations of the space $\wtH^{1,\gamma}_0(\Dm)$, Green's first identity without boundary term, and a weak setting where $\L_\gamma$ is bounded. Moving to the proof of Theorem \ref{thm:first}, it remains to show the extension of the Neumann trace to $W_\gamma^2$ by duality arguments, and the extension of Green's identities by density arguments. Finally, we prove Theorem \ref{thm:DNmap} in Section \ref{sec:DNmap}.

\subsection{Proof of Lemma \ref{lem:tgammab}} \label{sec:tgammab}

\paragraph{Case $-1<\gamma<0$.}

This case is the 'simplest' in that the $\wtH^{1,\gamma}$ space does not require the use of an auxiliary function $\phi_\gamma$, or equivalently, $\phi_\gamma = 1$ can be used, definition \eqref{eq:H1tilde} is equivalent to 
\begin{align}
    \A_\gamma\ni f,g\mapsto \dprod{f}{g}_{\wtH^{1,\gamma}} := \dprod{\sqrt{x}\ \partial_\rho f}{\sqrt{x}\ \partial_\rho g}_{L^2_\gamma} + \dprod{\rho^{-1} \partial_\omega f}{\rho^{-1} \partial_\omega g}_{L^2_\gamma} + (1+\gamma)^2 \dprod{f}{g}_{L^2_\gamma},
    \label{eq:H1tildem10}
\end{align}
hence the independence on $b$. Note that $\dir, \neu \colon \A_\gamma \to C^\infty(\Sm^1)$ defined in \eqref{eq:traces} take the simplified form
\begin{align*}
    \dir f = f|_{x=0}, \qquad \neu f = \lim_{x\to 0} (-x^{\gamma+1} \partial_\rho f), \qquad f\in \A_\gamma.
\end{align*}
Then \eqref{eq:G1} follows by sending $R\to 1$ in \eqref{eq:preG1}. 

\paragraph{Case $0\le \gamma<1$.} We first briefly explain what further treatment is required, explaining what happens for $\gamma\in (0,1)$, though the case $\gamma=0$ is similar. One may think about defining a space $\wtH^{1,\gamma}$ to be the completion of $C^\infty(\Dm)$ for the norm coming from the inner product \eqref{eq:H1tildem10}, though one may find that such a space does not have a good trace in the sense that, e.g. as a direct application of Corollary \ref{cor:H1dense}, for any $\gamma\in (0,1)$, $C_c^\infty(\Dm)$ is dense in $C^\infty(\Dm)$ for the norm coming from the inner product \eqref{eq:H1tildem10}. Since the indicial root $r=-\gamma$ is the most singular, one may think of taking the completion of $\A_\gamma$ with respect to the same norm, though another issue arises: the norm \eqref{eq:H1tildem10} of elements in $x^{-\gamma} C^{\infty}$ is in general not finite. In Sturm-Liouville theory, this situation is characteristic of the case of singular endpoint in the limit circle case, for which a systematic remedy in 1D can be found in, e.g., \cite[Sec. 6.9]{Behrndt2020}. We adapt these ideas to the present two-dimensional context.

We first make the function $\phi_\gamma$ introduced in Section \ref{sec:gammale1} explicit. For $x\in (0,1)$, we define
\begin{align}
    \phi_\gamma(x) := \left\{
    \begin{array}{cc}
	%1, & \gamma\in (-1,0), \\
	(-\gamma) \sum_{k=0}^\infty \frac{x^{k-\gamma}}{k-\gamma}, & \gamma\in (0,1), \\
	\log x- \log (1-x) - c_0, & \gamma=0,
    \end{array}
    \right. 
    \label{eq:phigamma}
\end{align}
where $c_0>\log(e^4-1)$ is a fixed constant. For any $\gamma\in [0,1)$, $\phi_\gamma$ is a radial solution of $\L_\gamma \phi_\gamma = (\gamma+1)^2 \phi_\gamma$ (this simplifies into $\partial_\rho (x^{\gamma+1}\rho\partial_\rho \phi_\gamma) = 0$), satisfying the relevant limits \eqref{eq:phigammaprops}. Moreover, there exists $x_\gamma>0$ such that $\phi_\gamma$ is nonzero on $(0,x_\gamma)$. If $x\in (0,1)$, $x_\gamma$ solves the implicit equation $\gamma \sum_{k=1}^\infty \frac{x^k}{k-\gamma} = 1$, and if $\gamma=0$, $x_0 = 1/(1+e^{-c_0})$. Set $b_\gamma := \sqrt{1-x_\gamma}$. 

Let us fix $\gamma\in [0,1)$, and write $\phi=\phi_\gamma$ for conciseness. Let us denote\footnote{The definition differs philosophically from \cite[Eq. (6.9.2) p443]{Behrndt2020} in a couple of ways: the $\sqrt{p}$ is not factored in; the derivative is the $\partial_\rho$ specific to our case.} 
\begin{align}
    N_\phi f := \phi\ \partial_\rho \left( \frac{f}{\phi} \right) = \partial_\rho f - \frac{\partial_\rho \phi}{\phi} f.
    \label{eq:Nphi}
\end{align}
Given $b\in (b_\gamma,1)$, let us define the bilinear form as in \eqref{eq:H1tilde} 
\begin{align}
    \begin{split}
	\ft_{\gamma,b}[f,g] &= \dprod{\sqrt{x} N_\phi f}{\sqrt{x} N_\phi g}_{x^\gamma, A_{b,1}} - b (x(b))^{\gamma+1} \frac{\partial_\rho \phi}{\phi}(b) \int_{C_b} f\bg  + \dprod{\sqrt{x} \partial_\rho f}{\sqrt{x} \partial_\rho g}_{x^\gamma, \Dm_{b}} \\	
	&\qquad + \dprod{\rho^{-1} \partial_\omega f}{\rho^{-1} \partial_\omega g}_{L^2_\gamma} + (\gamma+1)^2 \dprod{f}{g}_{L^2_\gamma},
    \end{split}    
    \label{eq:regform}
\end{align}
where for $0\le a<c\le 0$, we denote $A_{a,c}$ for the annulus $\{a<\rho<c\}$.

\begin{proof}[Proof of Lemma \ref{lem:tgammab}]
    We first derive some identities. We compute
    \begin{align*}
	\L_{\gamma} f\ \rho x^\gamma \bg = - \partial_\rho (x^{\gamma+1} \rho \partial_\rho f) \bg - \bg \rho^{-2}\partial_\omega ^2 f \rho x^\gamma + (\gamma+1)^2 f \bg \rho x^\gamma.
    \end{align*}
    The first term in the right-hand side is rewritten as
    \begin{align*}
	- \partial_\rho (x^{\gamma+1} \rho \partial_\rho f) \bg &= - \partial_\rho (x^{\gamma+1} \rho \partial_\rho (\phi f/\phi)) \bg \\
	&= - \partial_\rho (x^{\gamma+1} \rho (\partial_\rho \phi) f/\phi)) \bg - \bg \partial_\rho \left( x^{\gamma+1} \rho N_\phi f \right)\\ 
	&= - \cancel{\partial_\rho (x^{\gamma+1} \rho \partial_\rho \phi)} (f/\phi) \bg - x^{\gamma+1} \rho \frac{\partial_\rho \phi}{\phi} (N_\phi f) \bg - \bg \partial_\rho \left( x^{\gamma+1} \rho N_\phi f \right) \\
	%&\stackrel{\eqref{eq:phieq}}{=} -(\gamma+1)^2 fg \rho x^\gamma - x^{\gamma+1} \rho \frac{\partial_\rho \phi}{\phi} (N_\phi f) g - g \partial_\rho \left( x^{\gamma+1} \rho N_\phi f \right) \\
	&= - x^{\gamma+1} \rho \frac{\partial_\rho \phi}{\phi} (N_\phi f) \bg + \partial_\rho \bg x^{\gamma+1} \rho N_\phi f - \partial_\rho \left( \bg x^{\gamma+1} \rho N_\phi f \right)  \\ 
	&= (N_\phi \bg) x^{\gamma+1} \rho (N_\phi f)- \partial_\rho \left( \bg x^{\gamma+1} \rho N_\phi f \right). 
    \end{align*}
    Hence on an annulus $A_{a,b} = \{a<\rho<b\}$, we obtain the integration by parts formula
    \begin{align}
	\begin{split}
	    \dprod{\L_\gamma f}{g}_{x^\gamma, A_{a,b}} &= \dprod{\sqrt{x} N_\phi f}{\sqrt{x} N_\phi g}_{x^\gamma, A_{a,b}} + \dprod{\rho^{-1} \partial_\omega f}{\rho^{-1} \partial_\omega g}_{x^\gamma, A_{a,b}} \dots\\
	    & \qquad- \left[ x(\rho)^{\gamma+1} \rho \int_{C_\rho} \bg N_\phi f\right]_{\rho=a}^{\rho=b} + (\gamma+1)^2 \dprod{f}{g}_{x^\gamma, A_{a,b}}. 
	\end{split}
	\label{eq:IBPphi0}    
    \end{align}
    We now rewrite the $N_\phi f N_\phi \bg$ term: 
    \begin{align*}
	(N_\phi f) (N_{\phi} \bg) \rho x^{\gamma+1} &= \rho x^{\gamma+1} \left(\partial_\rho f - f \frac{\partial_\rho \phi}{\phi}\right)  \left(\partial_\rho \bg - \bg \frac{\partial_\rho \phi}{\phi}\right) \\
	&= \rho x^{\gamma+1} \partial_\rho f \partial_\rho \bg - \rho x^{\gamma+1} \partial_\rho (f \bg) \frac{\partial_\rho \phi}{\phi} + \rho x^{\gamma+1} f \bg \left( \frac{\partial_\rho \phi}{\phi} \right)^2 \\
	&= \rho x^{\gamma+1} \partial_\rho f \partial_\rho \bg -  \partial_\rho \left(\rho x^{\gamma+1} f \bg\frac{\partial_\rho \phi}{\phi}\right) + f \bg \partial_\rho \left( \rho x^{\gamma+1} \frac{\partial_\rho \phi}{\phi} \right)  + \rho x^{\gamma+1} f \bg \left( \frac{\partial_\rho \phi}{\phi} \right)^2 \\
	&\stackrel{\eqref{eq:phigammaprops}}{=} \rho x^{\gamma+1} \partial_\rho f \partial_\rho \bg -  \partial_\rho \left(\rho x^{\gamma+1} f \bg\frac{\partial_\rho \phi}{\phi}\right),
    \end{align*}
    and thus
    \begin{align}
	\dprod{\sqrt{x} N_\phi f}{\sqrt{x} N_\phi g}_{x^\gamma, A_{a,b}} &= \dprod{\sqrt{x} \partial_\rho f}{\sqrt{x} \partial_\rho g}_{x^\gamma, A_{a,b}} - \left[ \rho x^{\gamma+1} \frac{\partial_\rho \phi}{\phi} \int_{C_\rho} f \bg\right]_{\rho=a}^{\rho=b}.
	\label{eq:IBPphi}
    \end{align}
    
    \noindent {\bf Proof of (1).} If $\sqrt{1-x_\gamma}<a<b<1$, we compute
    \begin{align*}
	\ft_{\gamma,b}[f,g] &- \ft_{\gamma,a}[f,g] = -\dprod{\sqrt{x} N_\phi f}{\sqrt{x} N_\phi g}_{x^\gamma, A_{a,b}} - b (x(b))^{\gamma+1} \frac{\partial_\rho \phi}{\phi}(b) \int_{C_b} f\bg \\
	&\quad + a (x(a))^{\gamma+1} \frac{\partial_\rho \phi}{\phi}(a) \int_{C_a} f\bg + \dprod{\sqrt{x} \partial_\rho f}{\sqrt{x} \partial_\rho g}_{x^\gamma, A_{a,b}} \stackrel{\eqref{eq:IBPphi}}{=} 0,
    \end{align*}
    and hence $\ft_{\gamma,b}$ does not depend on $b\in (b_\gamma,1)$. 

    \noindent {\bf Proof of (2).} We note that if $f\in\A_\gamma$ with $\gamma\in(-1,0)$ or $f\in\A_{\gamma,D}$ with $\gamma\in[0,1)$ that
	\[\lim_{b\to 1}\left[b (x(b))^{\gamma+1} \frac{\partial_\rho \phi}{\phi}(b) \int_{C_b} |f|^2\right] = 0.\]
	Indeed, if $\gamma\in(-1,0)$ then $\partial_\rho\phi\equiv 0$, while otherwise
	\[b (x(b))^{\gamma+1} \frac{\partial_\rho \phi}{\phi}(b) = \begin{cases} O(x(b)^{\gamma}), & \gamma\in (0,1) \\ O\left(\frac{1}{|\log(x(b))|}\right), & \gamma=0\end{cases} = o(1)\]
	as $b\to 1$,  and $f\in\A_{\gamma,D}$ implies that $f$ is uniformly bounded as $x\to 0$, as $\A_{\gamma,D} = C^\infty(\Dm)+x^{1-\gamma}C^\infty(\Dm)$ if $\gamma\in(0,1)$, or $\A_{\gamma,D} = C^\infty(\Dm) + x(\log x)C^\infty(\Dm)$ for $\gamma=0$. As such, given that
	\[\lim_{b\to 1}{\|\sqrt{x} N_\phi f\|_{x^\gamma, A_{b,1}}^2} = 0\]
	since $\sqrt{x} N_\phi f\in L^2_\gamma$ near the boundary, it follows that we can take the limit as $b\to 1$ in \eqref{eq:regform} to obtain \eqref{eq:AgammaD} as desired. 

    \noindent {\bf Proof of (3).} Using that, for $b\in (0,1)$, 
    \begin{align*}
	\dprod{\L_\gamma f}{g}_{x^\gamma, \Dm_b} &= -\int_{C_b} x^{\gamma+1} \rho g\partial_\rho f + \dprod{\sqrt{x} \partial_\rho f}{\sqrt{x}\partial_\rho g}_{x^\gamma,\Dm_b} \\
	&\quad + \dprod{\rho^{-1}\partial_\omega f}{\rho^{-1} \partial_\omega g}_{x^\gamma, \Dm_b} + (\gamma+1)^2 \dprod{f}{g}_{x^\gamma, \Dm_b},
    \end{align*}
    combining with \eqref{eq:regform}, we arrive at
    \begin{align*}
	\ft_{\gamma}[f,g] &= \dprod{\sqrt{x} N_\phi f}{\sqrt{x} N_\phi g}_{x^\gamma, A_{b,1}} + \dprod{\rho^{-1} \partial_\omega f}{\rho^{-1} \partial_\omega g}_{x^\gamma, A_{b,1}} + (\gamma+1)^2 \dprod{f}{g}_{x^\gamma, A_{b,1}} \\
	&\qquad - \int_{C_b} \underbrace{( \rho x^{\gamma+1} \frac{\partial_\rho \phi}{\phi} f \bg - x^{\gamma+1} \rho \bg\partial_\rho f)}_{(\bg/\phi) W(f,\phi)} + \dprod{\L_\gamma f}{g}_{x^\gamma, \Dm_b}.
    \end{align*}
    Equation \eqref{eq:G1} then follows by sending $b\to 1$.

    \noindent {\bf Proof of (4).} We successively treat $\gamma\in (-1,0)$, $\gamma\in (0,1)$, and $\gamma=0$. 
    
    For $\gamma\in (-1,0)$, we note that since $\phi_{\gamma}\equiv 1$, we in fact have for any $f\in \A_\gamma$
    \begin{align}
	(f,f)_{\wtH^{1,\gamma}} = \|\sqrt{x}\partial_\rho f\|_{L^2_\gamma}^2 + \|\rho^{-1}\partial_\omega f\|_{L^2_\gamma}^2 + (\gamma+1)^2\|f\|_{L^2_\gamma}^2 \ge (1+\gamma)^2\|f\|_{L^2_\gamma}^2.
	\label{eq:coercm10}
    \end{align}
    The case $\gamma\in (0,1)$ is deduced from the case $\gamma\in (-1,0)$ by intertwining. First notice that if $f\in\A_\gamma$, then $x^\gamma f \in \A_{-\gamma}$. Assuming the following holds true
    \begin{align}
	\dprod{f}{g}_{\wtH^{1,\gamma}} = \dprod{x^\gamma f}{x^\gamma g}_{\wtH^{1,-\gamma}}, \qquad f, g\in \A_\gamma,
	\label{eq:linkH1}
    \end{align}
    the result then follows using the case $\gamma\in (-1,0)$: for any $f\in \A_\gamma$, 
    \begin{align*}
	(f,f)_{\wtH^{1,\gamma}} = (x^{-\gamma}f,x^{-\gamma}f)_{\wtH^{1,-\gamma}} \ge (1+(-\gamma))^2\|x^{-\gamma}f\|_{L^2_{-\gamma}}^2 = (1-|\gamma|)^2\|f\|_{L^2_{\gamma}}^2,
    \end{align*}
    where the inequality follows from \eqref{eq:coercm10}. To prove \eqref{eq:linkH1}, also notice that the differential intertwining property \eqref{eq:interxgamma} implies $\L_\gamma f = x^{-\gamma} \L_{-\gamma} (x^\gamma f)$ for all $f\in \A_\gamma$. We now compute 
    \begin{align*}
	\dprod{f}{g}_{\wtH^{1,\gamma}} &\!\!\stackrel{\eqref{eq:G1}}{=} \dprod{\L_\gamma f}{g}_{L^2_\gamma} - \dprod{\tau_\gamma^N f}{\tau_\gamma^D g}_{L^2(\Sm^1)} \\
	&= \dprod{x^{-\gamma} \L_{-\gamma} (x^\gamma f)}{x^{-\gamma} (x^\gamma g)}_{L^2_\gamma} - \dprod{\tau_\gamma^N f}{\tau_\gamma^D g}_{L^2(\Sm^1)} \\
	&\!\!\stackrel{\eqref{eq:intertrace}}{=} \dprod{\L_{-\gamma} (x^\gamma f)}{x^\gamma g}_{L^2_{-\gamma}} - \dprod{\tau_{-\gamma}^N (x^\gamma f)}{\tau_{-\gamma}^D (x^\gamma g)}_{L^2(\Sm^1)} \stackrel{\eqref{eq:G1}}{=} \dprod{x^\gamma f}{x^\gamma g}_{\wtH^{1,-\gamma}}.
    \end{align*}
    
    Finally the case $\gamma=0$ requires special care. Recall that $\phi_0(x) = \log x - \log (1-x) - c_0$ as defined in \eqref{eq:phigamma}. Let $x_0>0$ satisfy the property that $\phi_0(x)<0$ for all $x\in(0,x_0)$. Then, for any $b>b_0:=\sqrt{1-x_0}$, throwing away all first-order terms in \eqref{eq:H1tilde} gives the crude bound
    \begin{align*}
	(f,f)_{\wtH^{1,0}} \ge \|f\|_{L^2}^2 - \frac{bx(b)\partial_{\rho}\phi_0(x(b))}{\phi_0(x(b))}\int_{C_b}{|f|^2}.   
    \end{align*}
    Noting that $bx(b)\partial_{\rho}\phi_0(x(b)) = -2$ for all $b$, multiplying both sides by $- b\ \phi_0(x(b))/2$ and integrating from $b=b_0$ to $b=1$, we obtain
    \begin{align*}
	c (f,f)_{\wtH^{1,0}} \ge c\|f\|_{L^2}^2 - \|f\|_{A_{b_0,1}}^2, \quad \text{where} \quad c := \int_{b_0}^1{\frac{-\phi_0(x(b))}{2} b\,db} = \int_0^{x_0}{\frac{-\phi_0(x)}{4}\,dx}.    
    \end{align*}
    We can relate the constant $c$ to $c_0$ as follows: first note that the condition $\phi_0(x_0) = 0$ gives $x_0 = \frac{1}{1+e^{-c_0}}$. Integrating by parts, we then get
    \begin{align*}
	\int_0^{x_0}{-\phi_0(x)\,dx} = \underbrace{[-x \phi_0(x)]_{x\to 0}^{x=x_0}}_{=0} + \int_0^{x_0} x\phi_0'(x)\ dx = \int_0^{x_0} \frac{dx}{1-x} = - \log (1-x_0) = \log (1+e^{c_0}), 
    \end{align*}
    and hence $c = \frac{1}{4} \log(1+e^{c_0})$. Returning to the estimate, we arrive at
    \begin{align*}
	(f,f)_{\wtH^{1,0}} \ge \|f\|_{L^2}^2 - c^{-1}\|f\|_{A_{b_0,1}}^2 \ge (1-c^{-1})\|f\|_{L^2}^2,    
    \end{align*}
    and in particular $(\cdot,\cdot)_{\wtH^{1,0}}$ is positive definite if $c>1$, which is then equivalent to the condition $c_0>\log(e^4-1)$. 
    
	The proof of Lemma \ref{lem:tgammab} is complete. 
    \end{proof}

\subsection{Extension of the Dirichlet trace to $\wtH^{1,\gamma}$} \label{sec:Dext}

The first part of Theorem \ref{thm:first} consists in extending the Dirichlet trace to $\wtH^{1,\gamma}(\Dm)$, which we state as a separate result. 

\begin{theorem} \label{thm:qbt1}
    For any $\gamma\in (-1,1)$ and with $\wtH^{1,\gamma} (\Dm)$, $H_{(\gamma)}$ respectively defined in \eqref{eq:H1tildedef} and \eqref{eq:HgamS1}, the Dirichlet trace $\dir$ defined in \eqref{eq:traces} extends to a bounded, surjective operator 
    \begin{align*}
	\dir\colon \wtH^{1,\gamma} (\Dm) \to H_{(\gamma)}.
    \end{align*}
    The right inverse $R\colon H_{(\gamma)}\to \wtH^{1,\gamma}(\Dm)$ arises as the extension to $H_{(\gamma)}$ of a continuous operator $R\colon C^\infty(\Sm^1)\to \A_\gamma$.
\end{theorem}

In the case $\gamma\in (-1,0)$, the construction of $\wtH^{1,\gamma}$ is associated with operators with regular boundary points, in which case this is relatively straightforward. The construction for $\gamma\in (0,1)$ is associated with operators with singular points in the limit circle case, requiring a regularization approach, tying it with the previous family of cases $\gamma\in (-1,0)$. Finally, the case $\gamma=0$, linked with the case of double indicial roots and log-weighted Sobolev spaces, also requiring regularization, is treated separately.

\subsubsection{Some 1D trace estimates}\label{sec:1Dtraces}

Let $a>0$. Let us first prove a lemma that will be useful for the case $\gamma\in (-1,0)$. Define $\wtH^{1,\gamma} [0,a]$ to be the completion of $C^\infty( [0,a])$ for the norm 
\begin{align}
    \|f\|_{\wtH^{1,\gamma}[0,a]}^2 = \|\sqrt{x} \partial_x f\|_{L^2_\gamma[0,a]}^2 + (1+\gamma)^2 \|f\|_{L^2_\gamma[0,a]}^2,
    \label{eq:1DH1gamma}
\end{align}
where $\|f\|^2_{L^2_\gamma [0,a]} = \int_0^a |f(x)|^2 x^\gamma\ dx$. We prove the following trace estimate
\begin{lemma}\label{lem:1Dtrace}
    The evaluation map $C^\infty([0,a])\ni f\mapsto f(0)$ extends by density to $\wtH^{1,\gamma}[0,a]$ and we have the following estimate for all $\ell\in (0,a)$
    \begin{align}
	\ell^{\gamma} |f(0)|^2 \le (\gamma+1) \left( \ell^{-1} \|f\|_{L^2_\gamma [0,a]}^2 + \frac{2}{-\gamma} \|\sqrt{x} f'\|_{L^2_\gamma [0,a]}^2 \right).
	\label{eq:1Dtrace}
    \end{align}
\end{lemma}

\begin{proof}
    From the relation $f(0) = f(x) - \int_0^x f'(t)\ dt$, we deduce
    \begin{align*}
	|f(0)|^2 \le 2 |f(x)|^2 + 2 \left|\int_0^x f'(t)\ dt \right|^2 &\le 2 |f(x)|^2 + 2 \underbrace{\int_0^x t^{-\gamma-1}\ dt}_{x^{-\gamma}/(-\gamma)} \int_0^x t|f'(t)|^2\ t^\gamma dt.
    \end{align*}
    Now multiply by $x^\gamma$ and integrate from $0$ to $\ell$ to obtain
    \begin{align*}
	\frac{\ell^{\gamma+1}}{\gamma+1} |f(0)|^2 \le 2 \|f\|_{L^2_\gamma [0,a]}^2 + 2 \frac{\ell}{-\gamma} \|\sqrt{x} f'\|_{L^2_\gamma [0,a]}^2,
    \end{align*}
    equivalent to \eqref{eq:1Dtrace}. 
\end{proof}

Considering now a log-type weight that will be relevant to $\gamma=0$, let us now define $\wtH^{1,0}_{\log}[0,a]$ the closure of $C^\infty [0,a] + \frac{1}{\log x} C^\infty [0,a]$ for the norm 
\begin{align}
    \|f\|^2_{\wtH^{1,0}_{\log}} = \int_0^a (x |f'(x)|^2 + |f(x)|^2) \log^2 x\ dx.
    \label{eq:H1log}
\end{align}

\begin{lemma}\label{lem:H1log}
    The evaluation map $C^\infty[0,a] + \frac{1}{\log x} C^\infty[0,a]\ni f\mapsto f(0)$ extends to a bounded map on $\wtH^{1,0}_{\log} [0,a]$. More precisely, we have for every $\ell\in (0,a]$, 
    \begin{align}
	\frac{\int_0^\ell \log^2 x\ dx}{\int_0^\ell -\log x\ dx}|f(0)|^2 \le \frac{2}{\int_0^\ell -\log x\ dx} \int_0^a (\log x)^2 |f(x)|^2\ dx + 2 \int_0^a x\log^2 x |f'(x)|^2\ dx.
	\label{eq:H1logtrace}
    \end{align}    
\end{lemma}

\begin{proof} If $f\in C^\infty + \frac{1}{\log x} C^\infty$, then $f'$ is integrable near $0$ and we may write, for any $x\le a$, $f(0) = f(x)-\int_0^x f'(t)\ dt$. We then compute
    \begin{align*}
	|f(0)|^2 \le 2|f(x)|^2 + 2 \left| \int_0^x f'(t)\ dt\right|^2 \le 2|f(x)|^2 + 2 \underbrace{\int_0^x \frac{dt}{t\log^2 t}}_{-1/\log x} \int_0^a t\log^2 t |f'(t)|^2\ dt.
    \end{align*}
    Multiply by $\log^2 x$ and integrate from $0$ to $\ell\in (0,a)$ to obtain\footnote{one could be more explicit and use that $\int_0^\ell \log^2t\ dt= \ell ( (1-\log \ell)^2 + 1)$ and $\int_0^\ell \log t\ dt = \ell (\log \ell-1)$ but this may not be useful}
    \begin{align*}
	\int_0^\ell \log^2 x\ dx |f(0)|^2 \le 2\int_0^a (\log x)^2 |f(x)|^2\ dx + 2 \int_0^\ell (-\log x)\ dx \int_0^a t\log^2 t |f'(t)|^2\ dt.
    \end{align*} 
    Divide by $\int_0^\ell -\log x\ dx$ to obtain \eqref{eq:H1logtrace}.
\end{proof}

\begin{lemma}\label{lem:elleps}
    The equality $\int_0^{\ell} -\log x\ dx = \varepsilon$ gives rise to a function $\ell:[0,1]\to[0,1]$, continuous, increasing with $\ell(0) = 0$ and $\ell(1) = 1$, satisfying $\lim_{\varepsilon\to 0} \frac{\log(\ell(\varepsilon))}{\log \varepsilon} = 1$.
\end{lemma}

\begin{proof}
    The existence and increasing nature is obtained from inverse function theorem and differentiation. This equality can also be seen as $\ell (1-\log \ell) = \varepsilon$, which implies the obvious bound $\ell\le \varepsilon$, hence the zero limit. To obtain the asymptotics: we already have that $\log \ell\le \log \varepsilon$ by monotonicity of $\log$. In addition, for every $\alpha\in (0,1)$, for $\ell$ small enough, we have $1+ \log \frac{1}{\ell} \le C_\alpha \frac{1}{\ell^\alpha}$, and thus
    \begin{align*}
	\varepsilon = \ell (1-\log \ell) \le \ell^{1-\alpha} C_\alpha \quad \implies  \quad \log \varepsilon \le \log C_\alpha + (1-\alpha) \log \ell. 
    \end{align*}
    These inequalities imply $\frac{1}{1-\alpha} \le \liminf_{\varepsilon\to 0} \frac{\log\ell}{\log\varepsilon} \le \limsup_{\varepsilon\to 0} \frac{\log\ell}{\log\varepsilon} \le 1$ for every $\alpha>0$, hence the result follows.  
\end{proof}

\subsubsection{Proof of Theorem \ref{thm:qbt1}} \label{sec:qbt1proof}

The proof of Theorem \ref{thm:qbt1} treats, in order, the cases $\gamma\in (-1,0)$, $\gamma\in (0,1)$, and $\gamma=0$. 

\paragraph{Case $\gamma\in (-1,0)$.}

In a tubular neighborhood of the boundary $[0,a]_x \times \Sm^1_\omega$ (with $\partial \Dm = \{x=0\}$), one may expand a smooth function in the form $f = \sum_{n\in \Zm} f_n(x) e^{in\omega}$. Apply \eqref{eq:1Dtrace} to each $f_n$ with $\ell = a / \langle n \rangle^{2}$, sum over $n$ to obtain
\begin{align*}
    a^{\gamma} \sum_n \langle n\rangle^{-2\gamma} |f_n(0)|^2 \le (\gamma+1) \left( a^{-1} \sum_{n} \langle n\rangle^2 \|f_n\|_{L^2_\gamma [0,a]}^2 + \frac{2}{-\gamma} \sum_n \|\sqrt{x} f'_n\|_{L^2_\gamma [0,a]}^2  \right),
\end{align*}
Observing that the right-hand side is equivalent to the squared $\wtH^{1,\gamma}(\Dm)$ norm of $f$ and the left-hand side is the squared $H^{-\gamma}(\Sm^1)$-norm of $\omega\mapsto f(0,\omega)$, this boundedness estimate extends by density to $\wtH^{1,\gamma}(\Dm)$. 

To prove surjectivity, let us construct an explicit right-inverse for $\tau_\gamma^D$. Fix $g\in C_c^\infty( [0,a),\Rm)$ with $g(0) =1$ and consider the map 
\begin{align*}
    R\left(\sum_n a_n e^{in\omega}\right) = \sum_n a_n e^{in\omega} g(n^2 x). 
\end{align*}
It is easy to see that $R(C^\infty(\Sm^1))\subset C^\infty ([0,a]_x \times \Sm_\omega^1)$ and that $\dir R = id|_{C^\infty(\Sm^1)}$. To show that $R$ extend to a bounded map $H^{-\gamma}(\Sm^1) \to \wtH^{1,\gamma}  ([0,a]_x \times \Sm_\omega^1)$, it is enough to show that for $f_n(x) := g(n^2 x)$, $n^2 \|f_n\|_{L^2_\gamma([0,a])}^2 \sim \|\sqrt{x}f_n'(x)\|_{L^2_\gamma ([0,a])}^2 \lesssim n^{-2\gamma}$ with uniform constants in $n\in \Zm$. To see this, notice that $f_n'(x) = n^2 g'(n^2 x)$, then 
\begin{align*}
    n^2\int_0^a f^2(x) x^\gamma\ dx &\stackrel{u = n^2 x}{=} n^{-2\gamma}  \int_0^{n^2} g(u)^2 u^\gamma\ du = n^{-2\gamma}  \int_0^{a} g(u)^2 u^\gamma\ du, \\
    \int_0^a (f'(x))^2 x^{\gamma+1}\ dx &\stackrel{u = n^2 x}{=} n^{-2\gamma} \int_0^{n^2} g'(u)^2 u^{\gamma+1}\ du = n^{-2\gamma} \int_0^{a} g'(u)^2 u^{\gamma+1}\ du,
\end{align*}
This completes the proof of Theorem \ref{thm:qbt1} in the case $\gamma\in (-1,0)$.

\paragraph{Case $\gamma\in (0,1)$.}

The proof is based on combining the previous case with the convenient intertwining property \eqref{eq:interxgamma}. Fix $\gamma\in (0,1)$, and let $\psi$ be a radial function equal to $\phi_\gamma$ on $[0,x_\gamma/2)_x$, of class $C^\infty$ on $\Dm^{int}$ and bounded away from zero. The map $m_{1/\psi} \colon \A_\gamma \ni f \mapsto f/\psi \in \A_{-\gamma}$ extends by density into a homeomorphism $\wtH^{1,\gamma}(\Dm)\to \wtH^{1,-\gamma}(\Dm)$. Indeed, a direct calculation gives 
     \begin{align*}
	 \left\| \frac{f}{\psi}\right\|^2_{\wtH^{1,-\gamma}} &= \left \| \sqrt{x} \partial_\rho \frac{f}{\psi} \right\|^2_{L^2_{-\gamma}} +  \left\|\rho^{-1} \partial_\omega \frac{f}{\psi}\right\|^2_{L^2_{-\gamma}} + (1-\gamma)^2 \left\| \frac{f}{\psi} \right\|^2_{L^2_{-\gamma}} \\
	 &= \left \| \sqrt{x} \frac{x^{-\gamma}}{\psi} N_\psi f \right\|^2_{L^2_\gamma} +  \left\|\rho^{-1} \frac{x^{-\gamma}}{\psi} \partial_\omega f \right\|^2_{L^2_\gamma} + (1-\gamma)^2 \left\| \frac{x^{-\gamma}}{\psi} f \right\|^2_{L^2_\gamma}.
     \end{align*}
     The function $\frac{x^{-\gamma}}{\psi} \in C^\infty(\Dm)$ is bounded above and below by positive constants. The upper bound gives the boundedness of the map $m_{1/\psi}$ and the extension by density. The bound from below gives coercivity and by the open mapping theorem, the homeomorphism property. Combining this with the previous case, and noticing that $\tau_\gamma^D = \tau_{-\gamma}^D \circ m_{1/\psi}$ gives the result.

     \paragraph{The case $\gamma=0$.} While the singular case $\gamma\in (0,1)$ could make use of the space $\wtH^{1,-\gamma}$ associated with a 'regular' quadratic form, we need to construct the analogue of a 'pre-regularized' quadratic form with good trace estimates in the case $\gamma=0$. To this end, on a neighborhood of the boundary $\Sm^1_a = [0,a]_x \times \Sm^1_\omega$, let us define the norm on $C^\infty(\Sm^1_a)$
\begin{align}
    \|f\|^2_{\wtH^{1,0}_{\log}(\Sm^1_a)} := \|\sqrt{x}\log x\partial_x f\|_{L^2(\Sm_a^1)}^2 + \|\log x \partial_\omega f\|_{L^2(\Sm_a^1)}^2 + \|\log x f\|_{L^2(\Sm_a^1)}^2.
    \label{eq:H1logSa}
\end{align}
Writing $f(x,\omega) = \sum_{k\in \Zm} f_n(x)e^{in\omega}$ near $x=0$, this norm looks like
\begin{align*}
    \|f\|^2_{\wtH^{1,0}_{\log}(\Sm^1_a)} = 2\pi \sum_{k\in \Zm}\left( \|\sqrt{x}\log x\partial_x f_n\|_{L^2 [0,a]}^2 + (n^2 +1) \|\log x f_n\|_{L^2[0,a]}^2 \right).
\end{align*}
Upon defining $H_{(0)}$ as in \eqref{eq:HgamS1}, we now prove the following: 

\begin{lemma}\label{lem:tracegamma0}
    The evaluation map $C^1(\Sm^1_a) \ni f\mapsto f(0,\omega)\in C^\infty(\Sm^1)$ extends to a bounded, surjective trace map $\wtH^{1,0}_{\log} (\Sm^1_a)\to H_{(0)}$.
\end{lemma}

\begin{proof} To prove continuity, writing $f(x,\omega) = \sum_{k\in \Zm} f_n(x)e^{in\omega}$ near $x=0$, and use \eqref{eq:H1logtrace} on $f_n$ with $\ell_n\in (0,a]$ (assume $a\le 1$ WLOG) chosen such that 
    \begin{align*}
	\int_0^{\ell_n} -\log x\ dx = \frac{a}{\langle n\rangle^2}.    
    \end{align*}
    By Lemma \ref{lem:elleps}, $\ell_n$ exists, is unique, and 
    \begin{align*}
	1 = \lim_{n\to \infty} \frac{\log \ell_n}{\log (a/\langle n\rangle^2)} = \lim_{n\to \infty} \frac{\log \ell_n}{-2\log \langle n\rangle}.
    \end{align*}    
    Further, using l'Hopital's rule to show that $\lim_{t\to 0} \frac{\int_0^t \log^2 x\ dx}{\log t \int_0^t \log x\ dx} = 1$, we have
    \begin{align*}
	\frac{\int_0^{\ell_n} \log^2 x\ dx}{\int_0^{\ell_n} -\log x\ dx} \sim -\log\ell_n \sim 2\log\langle n \rangle, \qquad \text{ as } \quad n\to \infty.
    \end{align*}
    On to the surjectivity, it remains to construct a right inverse. To this end, fix $h\in C_c^\infty ( [0,a), \Rm)$ with $h(0) = 1$, and define $R\colon C^\infty(\Sm^1) \to C^\infty(\Sm_a^1) + \frac{1}{\log x} C^\infty(\Sm_a^1)$ as 
    \begin{align}
	R\colon \sum_{n\in \Zm} a_n e^{in\omega} \longmapsto \sum_{n\in \Zm} a_n e^{in\omega} \underbrace{h(n^2 x) \left( 1+\frac{2\log n}{\log x} \right)}_{f_n(x)}
	\label{eq:R0}
    \end{align}
    To show that $R\colon H_{(0)}\to \wtH^{1,0}_{\log} (\Sm^1_a)$ is bounded, it suffices to show that for $|n|$ large enough, the quantity
    \begin{align}
	\frac{1}{\log n} \left(\|\sqrt{x}\log x\partial_x f_n\|_{L^2[0,a]}^2 + n^2 \|\log x f_n\|_{L^2[0,a]}^2\right)
	\label{eq:temp}
    \end{align}
    is uniformly bounded by a constant independent of $n$. To that end, and upon noting that $f_n(x) = h(n^2 x) \frac{\log (n^2 x)}{\log x}$ and $f'_n(x) = n^2 \left( h'(n^2 x) \frac{\log (n^2 x)}{\log x} - 2\log n \frac{h(n^2 x)}{n^2 x\log^2 x} \right)$, we now compute
    \begin{align*}
	\int_0^a (n^2 f_n(x)^2 &+ x f'_n(x)^2) \log^2 x\ dx   \\
	\qquad &\!\!\!\!\!\stackrel{u=n^2 x}{=} \int_0^{an^2} \left(h^2(u) \log^2 u + u \left(h'(u) \log u -2\log n \frac{h(u)}{u \log (u/n^2)}\right)^2 \right)\ du \\
	\qquad &= \int_0^{an^2} \left((h^2(u) + uh'(u)^2) \log^2 u - 4\frac{\log n\ h'(u) h(u) \log u}{\log(u/n^2)} + 4 \log^2 n \frac{h^2(u)}{u \log^2(u/n^2)}\right) \ du \\
	\qquad &= I_{n,1} + I_{n,2} + I_{n,3}. 
    \end{align*}
    The term $I_{n,1}$ is asymptotic to a constant, thus $o(\log n)$. On to $I_{n,2}$, we rewrite
    \begin{align*}
	-4 \log n \frac{h'(u) h(u) \log u}{ \log(u/n^2)} = 2 \frac{h'(u) h(u) \log u}{1 - \frac{\log u}{2\log n}},
    \end{align*}
    which is dominated by the integrable function $2 |h'(u)h(u) \log u|$, and converges pointwise to $2h'(u) h(u) \log u$. By dominated convergence, $I_{n,2}$ is also asymptotically constant, thus $o(\log n)$. On to $I_{n,3}$, we write
    \begin{align*}
	I_{n,3} &= 4 \log^2 n \int_0^{n^2 a} \frac{h^2(u)}{u \log^2(u/n^2)} \ du \\
	&\!\!\!\!\!\stackrel{u = n^2 x}{=} 4 \log^2 n \int_0^a \frac{h^2(n^2 x)}{x \log^2 x}\ dx \\
	&= 4\log^2 n \left[\cancel{\int_0^a \frac{d}{dx} \left( \frac{h^2(n^2 x)}{-\log x} \right) dx} + 2 \int_0^a \frac{n^2 h'(n^2 x) h(n^2 x)}{\log x}\ dx\right] \\
	&\!\!\!\!\!\stackrel{u = n^2 x}{=} 8 \log^2 n \int_0^{n^2 a} \frac{h'(u) h(u)}{\log(u/n^2)}\ du \\
	&= - 4 \log n \int_0^{n^2 a} \frac{h'(u) h(u)}{1-\frac{\log u}{2\log n}}\ du.
    \end{align*}
    Similarly to the analysis of $I_{n,2}$, this is asymptotic to $-4\log n \int_0^{a} h'(u) h(u)\ du = 2\log n$. As a conclusion, \eqref{eq:temp} is uniformly bounded as claimed, and the map $R$ defined in \eqref{eq:R0} is bounded from $H_{(0)}$ to $\wtH^{1,0}_{\log} (\Sm^1_a)$. Lemma \ref{lem:tracegamma0} is proved. 
\end{proof}

On to the proof of Theorem \ref{thm:qbt1}, let $\psi$ be a radial function equal to $\phi_0$ on $[0,x_0/2)_x$, of class $C^\infty$ on $\Dm^{int}$ and bounded away from zero. One may view $\tau_0^D$ as the composition of
\begin{align*}
    \log x C^\infty(\Dm) + C^\infty(\Dm) \ni f\mapsto \frac{f}{\psi} \in C^\infty+ \frac{1}{\log x} C^\infty    
\end{align*}    
(which extends to an homeomorphism $\wtH^{1,0}(\Dm) \to \wtH^{1,0}_{\log} (\Dm)$), with the restriction map from Lemma \ref{lem:tracegamma0}.

The proof of Theorem \ref{thm:qbt1} is complete.

\subsection{Proof of Theorem \ref{thm:first}} \label{sec:prooffirstthm}

The following lemma gathers a few important preliminary facts. 

\begin{lemma}\label{lem:H1spaces} 

    (1) With the extended Dirichlet trace $\tau_D^\gamma$ defined in Theorem \ref{thm:qbt1}, we have
    \begin{align*}
	\wtH^{1,\gamma}_0 (\Dm) = \wtH^{1,\gamma}_D = \wtH^{1,\gamma}(\Dm) \cap \ker \tau_D^\gamma.
    \end{align*}

    (2) For every $f\in W^2_\gamma$, and $g\in \wtH^{1,\gamma}_0(\Dm)$, we have
    \begin{align}
	\dprod{\L_\gamma f}{g}_{L^2_\gamma} = \dprod{f}{g}_{\wtH^{1,\gamma}(\Dm)}.
	\label{eq:H10_id}
    \end{align}

    (3) The operator $\L_\gamma\colon \wtH^{1,\gamma}(\Dm) \to (\wtH^{1,\gamma}_0(\Dm))'$ is bounded.
\end{lemma}

\begin{proof}[Proof of Lemma \ref{lem:H1spaces}]
    
    Proof of (1). We first show that $\wtH^{1,\gamma}_0(\Dm) = \wtH^{1,\gamma}_D(\Dm)$. We use Lemma \ref{lem:Dcharac}, namely that the Dirichlet domain $\wtH^{2,\gamma}_D$ coincides with the domain of the Friedrichs extension. 
		
    By the construction of the Friedrichs extension, each element in the domain is the limit of functions in $C_c^\infty(\Dm)$ with respect to the norm $f\mapsto( \dprod{\L_\gamma f}{f}_{L^2_\gamma})^{1/2}$, which for functions in $C_c^\infty(\Dm)$ coincides with the $\wtH^{1,\gamma}$ norm. It follows that the Dirichlet domain $\wtH^{2,\gamma}_D(\Dm)$ is contained in $\wtH^{1,\gamma}_0(\Dm)$. On the other hand, since $C_c^\infty(\Dm)\subset\wtH^{2,\gamma}_D(\Dm)$, taking closures yields that $\wtH^{1,\gamma}_0(\Dm)$ is precisely the closure of $\wtH^{2,\gamma}_D(\Dm)$ with respect to the $\wtH^{1,\gamma}$ norm. However, since
    \begin{align*}
	\|f\|_{\wtH^{s,\gamma}_D}^2 = \sum_{n,k}{(n+1+|\gamma|)^{2s}|a_{n,k}|^2}\quad\text{if }f = \sum_{n,k}{a_{n,k} e_{n,k}}%\response{, \quad \text{where } e_{n,k}:= \left\{
%	\begin{array}{cc}    
%	    G_{n,k}^\gamma / \|G_{n,k}^\gamma\|_{L^2_\gamma}, & \gamma\ge 0, \\
%	    x^{-\gamma}G_{n,k}^{-\gamma} / \|x^{-\gamma}G_{n,k}^{-\gamma}\|_{L^2_\gamma}, & \gamma<0,
%	\end{array}
%        \right. }
    \end{align*}
    where $e_{n,k}$ are the $L^2_\gamma$-normalized multiples of \response{$G_{n,k}^\gamma$} if $\gamma\ge 0$ and \response{$x^{-\gamma}G_{n,k}^{-\gamma}$} if $\gamma<0$, and
    \[\|f\|_{\wtH^{1,\gamma}}^2 = (\L_\gamma f,f)_{L^2_\gamma} = \sum_{n,k}{(n+1+|\gamma|)^2|a_{n,k}|^2} = \|f\|_{\wtH^{1,\gamma}_D}^2,\]
    it follows that the closure of $\wtH^{2,\gamma}_D(\Dm)$ with respect to the $\wtH^{1,\gamma}$ norm is precisely $\wtH^{1,\gamma}_D(\Dm)$, as desired.
    
    Next, the inclusion $\wtH^{1,\gamma}_0(\Dm)\subset \wtH^{1,\gamma}(\Dm)\cap\ker\tau_D^\gamma$ follows from the continuity of $\tau_D^\gamma$ on $\wtH^{1,\gamma}(\Dm)$, as well as the fact that each element of $\wtH^{1,\gamma}_0(\Dm)$ is a limit of functions in $C_c^\infty(\Dm)$, all of which have trace zero. 
    
    Finally, to show that $\wtH^{1,\gamma}(\Dm)\cap\ker\tau_D^\gamma\subset\wtH^{1,\gamma}_0(\Dm)$, pick $f\in \wtH^{1,\gamma}(\Dm)\cap\ker\tau_D^\gamma$ and let $f_n \in \A_\gamma$ a sequence approximating $f$ in $\wtH^{1,\gamma}$. By continuity of $\tau_D^\gamma$, we must have that $\tau_D^\gamma f_n \to \tau_D^\gamma f = 0$ in $H_{(\gamma)}$.  With $R$ a right-inverse for $\tau_D^\gamma$ as in Theorem \ref{thm:qbt1}, then $R\tau_D^\gamma f_n\to 0$ in $\wtH^{1,\gamma}$ and hence $g_n := f_n - R \tau_D^\gamma f_n$ is a sequence of elements in $\A_{\gamma,D}$, converging to $f$ in $\wtH^{1,\gamma}$. Finally, by Corollary \ref{cor:H1dense}, $C_c^\infty(\Dm)$ is $\wtH^{1,\gamma}$-dense in $\A_{\gamma,D}$, and hence for each $n$, there is $h_n \in C_c^\infty(\Dm)$ such that $\|g_n-h_n\|_{\wtH^{1,\gamma}} <\frac{1}{n}$. The sequence $h_n$ converges in $\wtH^{1,\gamma}$ to $f$, which in turn belongs to $\wtH^{1,\gamma}_0$.

    Proof of (2). It is enough to show \eqref{eq:H10_id} for $f\in W^2_\gamma$ and $g\in \dot{C}^\infty$, since \eqref{eq:H10_id} then extends to $\wtH^{1,\gamma}_0$ by density. For $f\in W^{2}_\gamma\subset \wtH^{1,\gamma}$, let $f_n$ be a sequence in $\A_\gamma$ converging to $f$ in $\wtH^{1,\gamma}$, and fix $g\in \dot{C}^\infty$. We have, for every $n$, 
    \begin{align*}
	\dprod{f_n}{g}_{\wtH^{1,\gamma}} = \dprod{\L_{\gamma}f_n}{g}_{L^2_\gamma} = \langle \iota_\gamma \L_{\gamma}f_n, g\rangle = \langle {}^t\L_\gamma \iota_\gamma f_n, g \rangle. 
    \end{align*}
    Since $f_n\to f$ in $\wtH^{1,\gamma}$, hence in $L^2_\gamma$, $\iota_\gamma f_n\to f$ in $C^{-\infty}$, and by sequential continuity, ${}^t\L_\gamma \iota_\gamma f_n \to {}^t\L_\gamma \iota_\gamma f$ in $C^{-\infty}$. Hence we may send $n\to \infty$ in the above equality to obtain
    \begin{align*}
	\dprod{f}{g}_{\wtH^{1,\gamma}} = \langle {}^t\L_\gamma \iota_\gamma f, g \rangle, \qquad \forall g\in \dot{C}^\infty(\Dm).
    \end{align*}
    Finally, the assumption $\L_\gamma f\in L^2_\gamma$ gives that the right-side equals $\dprod{\L_\gamma f}{g}_{L^2_\gamma}$, and \eqref{eq:H10_id} follows for $f\in W^2_\gamma$ and $g\in \dot{C}^\infty$.

    Proof of (3). By \eqref{eq:H10_id}, we have the identity
    \begin{align*}
	|\dprod{\L_\gamma f}{g}_{L^2_\gamma}| = |\dprod{f}{g}_{\wtH^{1,\gamma}(\Dm)}| \le \|f\|_{\wtH^{1,\gamma}} \|g\|_{\wtH^{1,\gamma}}, \qquad f\in \A_\gamma, \qquad g\in \wtH^{1,\gamma}_0(\Dm).	
    \end{align*}
    This implies that $\L_\gamma \colon \A_\gamma \to (\wtH^{1,\gamma}_0)'$ is bounded, and that $\|\L_\gamma f\|_{(\wtH^{1,\gamma}_0)'} \le \|f\|_{\wtH^{1,\gamma}}$ for all $f\in \A_\gamma$. Since $\A_\gamma$ is dense in $\wtH^{1,\gamma}$, the latter inequality allows to extend $\L_\gamma$ to $\wtH^{1,\gamma}$ as a bounded $(\wtH^{1,\gamma}_0)'$-valued map.  
\end{proof}

We are now ready to prove Theorem \ref{thm:first}.

\begin{proof}[Proof of Theorem \ref{thm:first}] We first show that the Neumann trace extends into a bounded operator $\neu\colon W_\gamma^2\to H_{(\gamma)}'$. Following ideas in \cite[Sec. 4.4.4]{Helffer2013}, we combine Theorem \ref{thm:qbt1} with the second Green's identity. By virtue of Theorem \ref{thm:qbt1}, let $R\colon H_{(\gamma)} \to \wtH^{1,\gamma}$ be a bounded right-inverse for $\tau_\gamma^D$ arising from a continuous operator $R\colon C^\infty(\Sm^1) \to \A_\gamma$. For $f \in W_\gamma^2$, define the map 
    \begin{align}
	\psi_f (h) := \dprod{\L_{\gamma} f}{Rh}_{L^2_\gamma} - \dprod{f}{Rh}_{\wtH^{1,\gamma}}, \qquad h\in C^\infty(\Sm^1).
	\label{eq:defNeuext}
    \end{align}
    It is easily seen to satisfy an estimate of the form 
    \begin{align*}
	|\psi_f(h)| \le (\|\L_\gamma f\|_{L^2_\gamma} + \|f\|_{\wtH^{1,\gamma}}) \|Rh\|_{\wtH^{1,\gamma}} \le C (\|\L_\gamma f\|_{L^2_\gamma} + \|f\|_{\wtH^{1,\gamma}}) \|h\|_{H_{(\gamma)}},
    \end{align*}
    with $C$ the operator norm of $R$. By Riesz representation, this defines a unique element $\tau_\gamma^N f \in H_{(\gamma)}'$, which further satisfies the estimate $\|\tau_\gamma^N f\|_{H_{(\gamma)}'} \le C (\|\L_\gamma f\|_{L^2_\gamma} + \|f\|_{\wtH^{1,\gamma}})$. Moreover, for $f\in \A_\gamma$ and $h\in C^\infty (\Sm^1)$ so that $Rh\in \A_\gamma$, the first Green's identity \eqref{eq:G1} gives
    \begin{align*}
	\psi_f (h) = \dprod{\tau_\gamma^N f}{\tau_\gamma^D Rh}_{L^2(\Sm^1)} = \dprod{\tau_\gamma^N f}{h}_{L^2(\Sm^1)} = \langle \tau_\gamma^N f, h \rangle_{H'_{(\gamma)}, H_{(\gamma)}},
    \end{align*}
    hence the definition of $\neu$ on $W^2_\gamma$ extends the original definition \eqref{eq:traces} on ${\cal A}_\gamma$.

    On to extending Green's first identity to \eqref{eq:G1ext}, given $g\in \wtH^{1,\gamma}$, we decompose $g = R \dir g + g_0$, with $g_0 \in \wtH^{1,\gamma}_0$, and write, for $f\in W^2_\gamma$, 
    \begin{align*}
	\dprod{\L_\gamma f}{g}_{L^2_\gamma} - \dprod{f}{g}_{\wtH^{1,\gamma}} &=  \dprod{\L_\gamma f}{R \dir g}_{L^2_\gamma} - \dprod{f}{R \dir g}_{\wtH^{1,\gamma}} + \dprod{\L_\gamma f}{g_0}_{L^2_\gamma} - \dprod{f}{g_0}_{\wtH^{1,\gamma}} \\
	&\!\stackrel{\eqref{eq:defNeuext}}{=} \langle \tau_\gamma^N f, \tau_\gamma^D g \rangle_{H'_{(\gamma)}, H_{(\gamma)}} + \dprod{\L_\gamma f}{g_0}_{L^2_\gamma} - \dprod{f}{g_0}_{\wtH^{1,\gamma}},
    \end{align*}
    and the last two terms cancel out by virtue of \eqref{eq:H10_id}. Theorem \ref{thm:first} is proved. 
\end{proof}

\begin{remark} Similarly to \cite[Remark 8.2.5]{Behrndt2020}, the extension of $\neu$ could be made to the larger space $\widetilde{W}^{2}_\gamma = \{u\in \wtH^{1,\gamma}(\Dm),\ \L_\gamma u\in (\wtH^{1,\gamma}(\Dm))'\}$ into a bounded operator $\neu\colon \widetilde{W}^{2}_\gamma \to H_{(\gamma)}'$, and one would have a slightly more general first Green's identity
    \begin{align*}
	\langle \L_\gamma f, g\rangle_{(\wtH^{1,\gamma})', \wtH^{1,\gamma}} = \dprod{f}{g}_{\wtH^{1,\gamma}} + \langle \neu f, \dir g\rangle_{H'_{(\gamma)}, H_{(\gamma)}}, \qquad f\in \widetilde{W}^2_\gamma, \quad g\in \wtH^{1,\gamma}. 
    \end{align*}
    Upon skew-symmetrizing, Green's second identity generalizes to $f,g\in \widetilde{W}^2_\gamma$ in the obvious way.
\end{remark}

\subsection{DN map - Proof of Theorem \ref{thm:DNmap}} \label{sec:DNmap}

We end this section with the construction of the Dirichlet-to-Neumann map, i.e. the proof of Theorem \ref{thm:DNmap}.

\begin{proof}[Proof of Theorem \ref{thm:DNmap}] The construction of the DN map is done as follows. Consider $R_{\gamma}^D \colon H_{(\gamma)}\to \wtH^{1,\gamma}(\Dm)$ a continuous right inverse to $\tau_\gamma^D$ as in Theorem \ref{thm:qbt1}. For $f\in H_{(\gamma)}$, the construction of a unique solution $u_f\in W_\gamma^2 \cap \ker (\L_{\gamma}-\lambda)$ to 
    \begin{align}
	(\L_\gamma - \lambda)u = 0 \quad (\Dm), \qquad u|_{\Sm^1} = f % I don't think we need to change it here?
	\label{eq:upb}
    \end{align}
    is done as follows: Setting $\tilde f = R_{\gamma}^D f \in \wtH^{1,\gamma} (\Dm)$, write \response{$u_f = w + \tilde f$} for some unknown $w$, which in turn should solve
    \begin{align}
	(\L_\gamma -\lambda) w = - (\L_\gamma - \lambda) \tilde f \qquad (\Dm), \qquad w|_{\Sm^1} = 0,
	\label{eq:wpb}
    \end{align}
    where the right-hand side belongs to $(\wtH^{1,\gamma}_0(\Dm))'$ by Lemma \ref{lem:H1spaces}.(3), a space which coincides with $(\wtH^{1,\gamma}_D(\Dm))'=\wtH^{-1,\gamma}_D(\Dm)$ by Lemma \ref{lem:H1spaces}.(1). By setting up a weak formulation and invoking Riesz Representation Theorem on $\wtH^{1,\gamma}_D$, or simply using $(\L_{\gamma,D}-\lambda)^{-1}$ which is well-understood, this gives a unique solution to \eqref{eq:wpb} in $\wtH^{1,\gamma}_D$, given by 
    \begin{align*}
	w := - (\L_{\gamma,D}-\lambda)^{-1}(\L_\gamma - \lambda) \tilde f.	    
    \end{align*}
    Then \response{$u_f = w+\tilde f$} is a solution to \eqref{eq:upb} (which can also be proved to be unique since $\L_{\gamma,D}-\lambda$ is injective). Originally, it belongs to $\wtH^{1,\gamma}$, so in particular in $L^2_\gamma$, so \eqref{eq:upb} also implies that \response{$\L_\gamma u_f \in L^2_\gamma$}, and hence \response{$u_f\in W_\gamma^2\cap \ker (\L_{\gamma}-\lambda)$}.

    More succinctly, the DN map is then defined by 
    \begin{align*}
	\Lambda_\gamma (\lambda) f := \tau_\gamma^N u_f = \tau_\gamma^N (id - (\L_{\gamma,D}-\lambda)^{-1} (\L_{\gamma}-\lambda)) R_\gamma^D f,
    \end{align*}
    which boundedly lands into $H_{(\gamma)}'$ by Theorem \ref{thm:first}. The uniqueness of the solution to \eqref{eq:upb} makes it independent of the choice of right-inverse for $\tau_\gamma^D$.
\end{proof}

\section{Proof of Theorem \ref{thm:second}} \label{sec:proofsecond}

The idea is to follow the template of \cite[Theorem 8.4.1 p601]{Behrndt2020} to construct boundary triples for $\L_{\gamma,max}$ when $\gamma\in (-1,1)$. In proving Theorem \ref{thm:second}, we first prove in Section \ref{sec:tauNext} how to extend the Neumann trace on the Dirichlet Sobolev scale \eqref{eq:DSob}, this is formulated as Proposition \ref{prop:tauND} below, and makes crucial use of facts about generalized Zernike polynomials. In Section \ref{sec:tauDtilde}, we then show how to extend the Dirichlet trace, a result formulated in Proposition \ref{prop:tauDext}. Finally, we complete the proof of Theorem \ref{thm:second} in Section \ref{sec:G2ext}, first proving Lemma \ref{lem:directsum} on the decompositions of the maximal domain, then using these decompositions to extend Green's second identity to \eqref{eq:G2extmax}.

\subsection{Extension of the Neumann trace}\label{sec:tauNext}

\begin{proposition}\label{prop:tauND}
    For $\gamma\in (-1,1)$ and any $s>1+|\gamma|$, the map $\neu$ defined in \eqref{eq:traces} extends to a bounded, surjective map
    \begin{align}
	\tau^N_\gamma \colon \wtH^{s,\gamma}_D(\Dm)\to H^{s-1-|\gamma|}(\Sm^1),
	\label{eq:tauND}
    \end{align}
    with a right inverse arising as the extension to \response{$H^{s-1-|\gamma|} (\Sm^1)$} of a continuous operator $C^\infty(\Sm^1)\to C^\infty(\Dm)$. In particular, for $s=2$, where $\wtH^{2,\gamma}_D(\Dm) = \dom(\L_{\gamma,D})$
    \begin{align}
	\tau_\gamma^N \colon \dom(\L_{\gamma,D}) \to H^{1-|\gamma|}(\Sm^1)
	\label{eq:tau2ND}
    \end{align}
    is a bounded and surjective map, and we have the further characterization
    \begin{align}
	\dom(\L_{\gamma,min}) = \wtH^{2,\gamma}_D \cap \ker \neu.
	\label{eq:Lmin}
    \end{align}
\end{proposition}

\begin{proof} (Case $0\le \gamma < 1$) \response{With $\{G_{n,k}^\gamma\}_{n,k}$ defined in \eqref{eq:Gnkgamma}}, the main two fundamental properties needed are
    \begin{align}
	\tau_\gamma^N (\response{G_{n,k}^{\gamma}}) &= c_\gamma \response{G_{n,k}^{\gamma}}|_{\Sm^1} = c_\gamma e^{i(n-2k)\beta} \label{eq:ppty1}\\
	\|\response{G_{n,k}^{\gamma}}\|^2_{L^2_{\gamma}} &= \frac{\pi}{n+\gamma+1} \frac{(n-k)!\gamma!k!\gamma!}{(k+\gamma)!(n-k+\gamma)!} =: (n_{n,k}^{\gamma})^2, \label{eq:ppty2}
    \end{align} 
    where we write for short $x! := \Gamma(x+1)$, and where $c_\gamma = 2\gamma$ for $\gamma\in (0,1)$, and $c_0 = -2$. The first equality uses \cite[Eq (2.10)]{Wuensche2005}. Hence, for $u\in \wtH_D^{s,\gamma}$\response{, of the form $u = \sum_{n=0}^\infty \sum_{k=0}^n u_{n,k} \frac{G_{n,k}^\gamma}{n_{n,k}^\gamma}$ with $\|u\|^2_{\wtH_D^{s,\gamma}}=\sum_{n=0}^\infty \sum_{k=0}^n (n+1+\gamma)^{2s} |u_{n,k}|^2 < \infty$}, we have 
    \begin{align*}
	\tau_\gamma^N u = c_\gamma \sum_{n,k} \frac{u_{n,k}}{n_{n,k}^\gamma} e^{i(n-2k)\beta} = c_\gamma \sum_{m\in \Zm} e^{im\beta} [\tau_\gamma^N u]_m, \qquad [\tau_\gamma^N u]_m := \sum_{n-2k=m} \frac{u_{n,k}}{n_{n,k}^{\gamma}}. 
    \end{align*}

    Then 
    \begin{align}
	\|\tau_\gamma^N u\|_{H^{s-\gamma-1}}^2 &= 2\pi c_\gamma^2  \sum_{m\in \Zm} \langle m\rangle^{2s-2\gamma-2} \left|[\tau_\gamma^N u]_m\right|^2 \nonumber	\\
	&\le 2\pi c_\gamma^2  \sum_{m\in \Zm} \langle m\rangle^{2s-2\gamma-2} \left(\sum_{n-2k =m} (n+1+\gamma)^{2s} |u_{n,k}|^2 \right) \left( \sum_{n-2k =m}  \frac{(n+1+\gamma)^{-2s}}{(n_{n,k}^{\gamma})^2} \right)
	\label{eq:tau_n-cs}
    \end{align}
    and the proof is complete if we can show that $C_{m; s,\gamma} := \langle m\rangle^{2s-2\gamma-2}\sum_{n-2k =m}  \frac{(n+1+\gamma)^{-2s}}{(n_{n,k}^{\gamma})^2} $ is uniformly bounded for all $m\in \Zm$. Note that $C_{m; s,\gamma} = C_{-m;s,\gamma}$ so it enough to check it for $m\ge 0$, for which we have
    \begin{align*}
	C_{m; s,\gamma} &=  \langle m\rangle^{2s-2\gamma-2} \sum_{\ell\ge 0} \frac{(m+2\ell + 1+\gamma)^{-2s}}{(n_{m+2\ell,\ell}^{\gamma})^2} \\	
	&=  \frac{\langle m\rangle^{2s-2\gamma-2}}{\pi (\gamma!)^2} \sum_{\ell\ge 0} (m+2\ell + 1+\gamma)^{1-2s} \frac{(m+\ell+\gamma)!}{(m+\ell)!} \frac{(\ell+\gamma)!}{\ell!}.
    \end{align*}

    To this end, we state the following estimate, which is relegated to Appendix \ref{sec:lemCmbounds}: 
    \begin{lemma}\label{lem:Cmbounds}
	There exist $C',C''>0$ such that $C'\le C_{m; s,\gamma} \le C''$ for all $m\ge 0$.
    \end{lemma}
    Using Lemma \ref{lem:Cmbounds}, we thus see that
    \begin{align*}
	\|\tau_\gamma^N u\|_{H^{s-\gamma-1}}^2 \le 2\pi c_\gamma^2 \sum_{m\in \Zm} C_{m; s,\gamma} \left(\sum_{n-2k =m} (n+1+\gamma)^{2s} |u_{n,k}|^2 \right) \le 2\pi c_\gamma^2 \left(\sup_{m\in\Zm}{C_{m; s,\gamma}}\right)\|u\|_{\wtH^{s,\gamma}_D(\Dm)}^2,	
    \end{align*}
    thus establishing the forward trace estimate.
    	
    To establish surjectivity, it suffices to construct a bounded right inverse $R:H^{s-\gamma-1}(\Sm^1)\to\wtH^{s,\gamma}_D(\Dm)$. We construct the following map: 
    \begin{align}
	R\left(\sum_{m\in\Zm}{a_me^{im\omega}}\right) = \frac{1}{c_\gamma} \sum_{m\in\Zm}{\frac{a_m}{1+|m|}\sum_{n-2k=m, n\le 3|m|}{ \response{G_{n,k}^\gamma}}}.
	\label{eq:Rneu}
    \end{align}
    Note that the set $\{(n,k)\,:\,n-2k=m,\ 0\le k\le n\}$ can be parametrized by
    \[n=|m|+2\ell ,\quad k = \frac{|m|-m}{2}+\ell = \max(0,-m) + \ell, \quad \ell\ge 0,\]
    so the inner sum is over $1+|m|$ elements. Specifically, $R$ sends $e^{im\omega}$ to the average over the $|m|+1$ weighted Zernike polynomials of lowest degree whose trace equals $e^{im\omega}$. In particular, that $R$ is a right-inverse directly follows from property \eqref{eq:ppty1}. 

    Moreover we have
    \begin{align*}
	\left\| R\left(\sum_{m\in\Zm}{a_me^{im\omega}}\right)\right\|_{\wtH^{s,\gamma}_D(\Dm)}^2 &=  \frac{1}{c_\gamma^2} \sum_{m\in\Zm}{\frac{|a_m|^2}{(1+|m|^2)}\sum_{n-2k=m,n\le 3|m|}{ (n+1+\gamma)^{2s}(n_{n,k}^{\gamma})^2}} \\
	&=  \frac{1}{c_\gamma^2} \sum_{m\in\Zm}{\langle m\rangle^{2s-2\gamma-2}|a_m|^2\frac{\langle m\rangle^2}{(1+|m|)^2}C_m'} \\
	&\le \left(\sup_{m\in\Zm}{\frac{\langle m\rangle^2C_m'}{c_\gamma^2(1+|m|)^2}}\right) \left\|\sum_{m\in \Zm} a_m e^{im\omega}\right\|^2_{H^{s-1-\gamma}},
    \end{align*}
    where we have defined 
    \[C_m' := \langle m\rangle^{-2s+2\gamma}\sum_{n-2k=m,n\le 3|m|}{(n+1+\gamma)^{2s}(n_{n,k}^{\gamma})^2}.\]
    Thus, that $R\colon H^{s-1-|\gamma|}(\Sm^1)\to \wtH^{s,\gamma}_D (\Dm)$ is bounded will follow upon showing that $C_m'$ is uniformly bounded in $m$. This is implied by the following asymptotics
    \begin{align}
	\sum_{n-2k=m,n\le 3|m|}{(n+1+\gamma)^{2s}(n_{n,k}^{\gamma})^2} = O(\langle m\rangle^{2s-2\gamma}), \qquad \text{as}\quad |m|\to \infty,
	\label{eq:Cmasymp}
    \end{align}
    which we prove now. As before, $C_{-m}' = C_m'$, so it suffices to assume $m\ge 0$, in which case we parametrize $n = m+2\ell$ and $k=\ell$, with $0\le\ell\le m$. Similarly as before we assume $m$ is sufficiently large, say $m>\gamma$. Then the sum then becomes
    \begin{align*}
	\sum_{n-2k=m,n\le 3|m|}{(n+1+\gamma)^{2s}(n_{n,k}^{\gamma})^2} \stackrel{\eqref{eq:ppty2}}{=} \frac{\pi}{(\gamma!)^2}\sum_{\ell=0}^m{(m+2\ell+1+\gamma)^{2s-1}\frac{(m+\ell)!}{(m+\ell+\gamma)!}\frac{\ell!}{(\ell+\gamma)!}}.	
    \end{align*}
    Using the asymptotic $\lim_{x\to\infty}{\frac{x!}{(x+\gamma)!(x+1)^{-\gamma}}} = 1$, and using that $m+2\ell+1+\gamma\le 4(m+1)$ and $2s-1\ge 1+2\gamma\ge 1$, the above sum is bounded above by some multiple of 
    \begin{align*}
	\sum_{\ell=0}^m{(m+2\ell+1+\gamma)^{2s-1}(m+\ell+1)^{-\gamma}(\ell+1)^{-\gamma}} \lesssim (m+1)^{2s-\gamma-1}\sum_{\ell=0}^m{(\ell+1)^{-\gamma}}.
    \end{align*}
    Finally, the function $x\mapsto (x+1)^{\gamma}$ is decreasing, so we can estimate
    \[\sum_{\ell=0}^m{(\ell+1)^{-\gamma}} = 1+\sum_{\ell=1}^m{(\ell+1)^{-\gamma}}\le 1+\int_0^m{(x+1)^{-\gamma}\,dx} =1+\frac{(m+1)^{1-\gamma}-1}{1-\gamma}.\]
    It follows that
    \[\sum_{n-2k=m,n\le 3|m|}{(n+1+\gamma)^{2s}(n_{n,k}^{\gamma})^2} \lesssim (m+1)^{2s-\gamma-1}\sum_{\ell=0}^m{(\ell+1)^{\gamma}}\lesssim (m+1)^{2s-2\gamma},\]
    i.e. \eqref{eq:Cmasymp} follows.
   
    (Case $\gamma<0$) Notice that the map $m_{x^\gamma}\colon x^{-\gamma}C^\infty(\Dm) \ni f\mapsto x^\gamma f\in C^\infty(\Dm)$ extends into an isometry $\wtH^{s,\gamma}_D \mapsto \wtH_D^{s,-\gamma}$, and that we have $\tau_\gamma^N = \tau_{-\gamma}^N \circ m_{x^\gamma}$. The result for $\gamma\in (-1,0)$ follows.  

    Finally, we prove the characterization \eqref{eq:Lmin} of the minimal domain. The second Green's identity \eqref{eq:G2} extends, for $f\in\wtH^{2,\gamma}_D(\Dm)$, to the equation
    \begin{equation}
	\label{eq:G2D}
	(\L_\gamma f,g)_{L^2_\gamma}-(f,\L_\gamma g)_{L^2_\gamma} = (\tau^N_\gamma f,\tau^D_\gamma g)_{L^2(\Sm^1)},\quad g\in\A_\gamma.
    \end{equation}
    Since $\wtH^{2,\gamma}_D(\Dm)$ is the domain of the Dirichlet realization $\L_{\gamma,D}$, it follows the minimal domain is contained in $\wtH^{2,\gamma}_D(\Dm)$. As such, let $f\in\text{dom }(\L_{\gamma,min})$; then $f\in \wtH^{2,\gamma}_D(\Dm)$. Moreover, since $\L_{\gamma,min} = (\L_{\gamma,max})^*$, it follows that for all $g\in \L_{\gamma,max}$ we have
    \[(\L_\gamma f,g)_{L^2_\gamma}-(f,\L_\gamma g)_{L^2_\gamma} = 0.\]
    In particular this holds for all $g\in\A_\gamma$, in which case this combined with \eqref{eq:G2D} gives
    \[(\tau^N_\gamma f,\tau^D_\gamma g)_{L^2(\Sm^1)} = 0\text{ for all }g\in\A_\gamma,\]
    i.e. $\tau^N_\gamma f$ is orthogonal (in $L^2(\Sm^1)$) to $\dir(\A_\gamma) = C^\infty(\Sm^1)$. By density of the latter, this forces $\tau^N_\gamma f\equiv 0$. It follows that $\text{dom }(\L_{\gamma,\min})\subset \wtH^{2,\gamma}_D(\Dm)\cap \ker\tau^N_\gamma$.
    
    To show the other inclusion, we will use density arguments. First considering the case $\gamma\in[0,1)$, first notice that by Theorem \ref{thm:densities}.(ii), we have $xC^\infty (\Dm) \subset \overline{C_c^\infty(\Dm)}^{n_\gamma} =\text{dom}(\L_{\gamma,min})$, and upon taking $n_\gamma$-closures, $\overline{xC^\infty(\Dm)}^{n_\gamma} \subset \text{dom}(\L_{\gamma,min})$. Then $\wtH^{2,\gamma}_D(\Dm)\cap \ker\tau^N_\gamma\subset \text{dom}(\L_{\gamma,min})$ will hold provided that 
    \begin{align}
        \wtH^{2,\gamma}_D(\Dm)\cap \ker\tau^N_\gamma= \overline{xC^\infty(\Dm)}^{n_\gamma}.
    \label{eq:firsteq}
    \end{align}
   The inclusion $\supset$ in \eqref{eq:firsteq} is clear since the space on the left is $n_\gamma$-closed and contains $xC^\infty(\Dm)$. On to the inclusion $\subset$, for any $f\in \wtH^{2,\gamma}_D(\Dm)$, there is a sequence $f_n\in C^\infty(\Dm)$ with $f_n\to f$ in $\wtH^{2,\gamma}_D(\Dm)$. If $f$ also satisfies $\tau^N_\gamma f\equiv 0$, then $\tau^N_\gamma f_n\to 0$ in $H^{1-|\gamma|}(\Sm^1)$. With $R$ the right inverse for $\tau_\gamma^N$ defined in \eqref{eq:Rneu}, $R\tau_\gamma^N f_n\in C^\infty(\Sm^1)$ and $R\tau_\gamma^N f_n \to 0$ in $\wtH^{2,\gamma}_D(\Dm)$. Considering 
    \[g_n = f_n - R\tau^N_\gamma f_n \in C^\infty(\Dm),\]
    since $\tau^N_\gamma g_n = \tau^N_\gamma f_n - \tau^N_\gamma R\tau^N_\gamma f_n = 0$, $g_n$ in fact belongs to $xC^\infty(\Dm)$ and converges to $f$ in $\wtH^{2,\gamma}_D(\Dm)$. This gives \eqref{eq:firsteq}. 

For $\gamma\in(-1,0)$, we follow a similar strategy to show that $\wtH^{2,\gamma}_D(\Dm)\cap \ker\tau^N_\gamma = \overline{x^{1-\gamma}C^\infty(\Dm)}^{n_\gamma}$, and noting that $C_c^\infty(\Dm)$ is $n_\gamma$-dense in $x^{1-\gamma}C^\infty(\Dm)$ for all $\gamma\in(-1,0)$ (in fact for all $\gamma<1$).    
\end{proof}

From this, we can also show:
\begin{lemma}
    \label{lem:agammadense}
    For $\gamma\in(-1,1)$,  we have $\dom(\L_{\gamma,max}) = \overline{\A_\gamma}^{n_\gamma}$. 
\end{lemma}
\begin{proof}
    Consider the operator $\L_{\gamma,\A_\gamma}$ whose domain is $\A_\gamma$. It suffices to show that
    \[\text{dom }((\L_{\gamma,\A_\gamma})^*)\subset\text{dom }(\L_{\gamma,min}),\]
    since by taking adjoints we would have
    \[\text{dom }(\L_{\gamma,max}) = \text{dom }((\L_{\gamma,min})^*)\subset \text{dom }((\L_{\gamma,\A_\gamma})^{**}) = \overline{\A_{\gamma}}^{n_\gamma}.\]
    Note that
    \[\wtH^{2,\gamma}_D(\Dm)\subset\overline{\text{dom }(\L_{\gamma,\A_\gamma})}^{n_\gamma}\]
    by taking closures on either $C^\infty(\Dm)$ or $x^{-\gamma}C^\infty(\Dm)$, and hence taking adjoints gives
    \[\text{dom }((\L_{\gamma,\A_\gamma})^*)\subset \wtH^{2,\gamma}_D(\Dm).\]
    As such, suppose $f\in\text{dom }((\L_{\gamma,\A_\gamma})^*)$; then $f\in\wtH^{2,\gamma}_D(\Dm)$. For all $g\in\A_\gamma$, equation \eqref{eq:G2D} gives that
    \[(\L_\gamma f,g)_{L^2_\gamma}-(f,\L_\gamma g)_{L^2_\gamma} = (\tau^N_\gamma f,\tau^D_\gamma g)_{L^2(\Sm^1)}.\]
    But we also have $(f,\L_\gamma g) = (f,(\L_{\gamma,\A_\gamma})g) = ((\L_{\gamma,\A_\gamma})^*f,g) = (\L_\gamma f,g)$, so we have
    \[(\tau^N_\gamma f,\tau^D_\gamma g)_{L^2(\Sm^1)} = 0\text{ for all }g\in\A_\gamma.\]
    Similar reasoning as above yields $\tau^N_\gamma f\equiv 0$, except now that the characterization \eqref{eq:Lmin} has been proven, we can conclude that $f\in\text{dom }(\L_{\gamma,min})$, as desired.
\end{proof}

\subsection{Extension of the Dirichlet trace}\label{sec:tauDtilde}

\begin{proposition} \label{prop:tauDext}
    For $\gamma\in (-1,1)$, the restriction to $W_\gamma^2$ of the Dirichlet trace $\tau_{\gamma}^D$ defined in Theorem \eqref{thm:qbt1} extends to a bounded, surjective map
    \begin{align*}
	\tilde\tau_\gamma^D \colon \dom(\L_{\gamma,max}) \to H^{-1+|\gamma|} (\Sm^1).
    \end{align*}
    Moreover, we have, for any $f\in \dom (\L_{\gamma,max})$ and $g\in \wtH^{2,\gamma}_D(\Dm)$, 
    \begin{align}
	\dprod{\L_\gamma f}{g}_{L^2_\gamma} - \dprod{f}{\L_\gamma g}_{L^2_\gamma} = - \langle \tilde\tau^D_\gamma f, \tau_\gamma^N g\rangle_{H^{-1+|\gamma|}, H^{1-|\gamma|}}.
	\label{eq:BTid}
    \end{align}    
\end{proposition}

\begin{proof}[Proof of Proposition \ref{prop:tauDext}] Let $R\colon H^{1-|\gamma|}(\Sm^1) \to \wtH^{2,\gamma}_D(\Dm)$ a right inverse for the map \eqref{eq:tau2ND}, in particular $\tau_\gamma^N (Rh) = h$ for all $h\in H^{1-|\gamma|}(\Sm^1)$. For $f\in \dom(\L_{\gamma,max})$, define the functional $\phi_f$ on $H^{1-|\gamma|}(\Sm^1)$ by 
    \begin{align}
	\phi_f (h) = - \dprod{\L_\gamma f}{Rh}_{L^2_\gamma} + \dprod{f}{\L_\gamma Rh}_{L^2_\gamma}, \qquad h\in H^{1-|\gamma|}(\Sm^1).
	\label{eq:Dext}
    \end{align}
    The map $\phi_f$ obviously satisfies 
    \begin{align*}
	|\phi_f(h)| &\le \|\L_\gamma f\|_{L^2_\gamma} \|Rh\|_{L^2_\gamma} + \|f\|_{L^2_\gamma} \|\L_\gamma Rh\|_{L^2_\gamma} \\
	&\le (\|f\|_{L^2_\gamma} + \|\L_\gamma f\|_{L^2_\gamma}) (\|Rh\|_{L^2_\gamma} + \|\L_\gamma Rh\|_{L^2_\gamma} ) \le C  (\|f\|_{L^2_\gamma} + \|\L_\gamma f\|_{L^2_\gamma}) \|h\|_{H^{1-|\gamma|}(\Sm^1)},
    \end{align*}
    by boundedness of $R$. By Riesz representation, this defines a unique element $\tilde\tau_\gamma^D f$ in $H^{-1+|\gamma|}(\Sm^1)$. %Note that if $f\in \wtH^{2,\gamma}_D$, the definition \eqref{eq:Dext} vanishes identically so $\tilde\tau_\gamma^D =0$ on $\wtH^{2,\gamma}_D$. 
    If $f\in W_\gamma^2$, then \eqref{eq:G2ext} implies that $\phi_f (h) = \langle \neu (Rh), \dir f\rangle_{H'_{(\gamma)}, H_{(\gamma)}}$ where, since $\tau_\gamma^D f \in H_{(\gamma)} \subset H^{-1+|\gamma|}$, and since $\tau_\gamma^N (Rh) = h \in H^{1-|\gamma|}$, the latter pairing also equals to $\langle \tau_\gamma^D f, h \rangle_{H^{-1+|\gamma|}, H^{1-|\gamma|}}$. In particular, if $f\in W^2_\gamma$, then $\tilde\tau_\gamma^D f = \tau_\gamma^D f$ as elements of $H^{-1+|\gamma|}$. 

    On to proving \eqref{eq:BTid}, pick $f\in \dom(\L_{\gamma,max})$ and $g\in \wtH^{2,\gamma}_D$. Write $g = R \tau_\gamma^N g + g_{00}$ where $\tau_\gamma^N g_{00} = 0$ hence $g_{00} \in \dom(\L_{\gamma,min})$. Then, writing $\dprod{\cdot}{\cdot}$ for the $L^2_\gamma$ inner-product for conciseness,
    \begin{align*}
	\dprod{\L_\gamma f}{g} - \dprod{f}{\L_\gamma g} &= \dprod{\L_\gamma f}{R \tau_\gamma^N g} - \dprod{f}{\L_\gamma R \tau_\gamma^N g} + \dprod{\L_\gamma f}{g_{00}} - \dprod{f}{\L_\gamma g_{00}} \\
	&\!\stackrel{\eqref{eq:Dext}}{=} - \langle \tilde\tau^D_\gamma f, \tau_\gamma^N g\rangle_{H^{-1+|\gamma|}, H^{1-|\gamma|}} + \dprod{\L_\gamma f}{g_{00}} - \dprod{f}{\L_\gamma g_{00}}
    \end{align*}
    and the last two terms cancel out since $\L_{\gamma,max} = (\L_{\gamma,min})^*$. 

    To show surjectivity of $\tilde\tau_\gamma^D$, let us now construct a right inverse. As a combination of Proposition \ref{prop:tauND} and \eqref{eq:Linv}, the operator $\Upsilon:= \neu \L_{\gamma,D}^{-1} \colon L^2_\gamma \to H^{1-|\gamma|}(\Sm^1)$ is continuous and onto (with $R$ defined in this proof, $\L_{\gamma}R$ is a right inverse for $\Upsilon$), hence the transpose $\Upsilon'\colon H^{-1+|\gamma|}(\Sm^1)\to L^2_\gamma$ is bounded. 
    
    We first claim that $\Upsilon'$ is valued in $\fN_0(\L_{\gamma,max})$, whose topology is captured by $L_\gamma^2$. This makes $\Upsilon'$ a continuous $\dom(\L_{\gamma,max})$-valued map. To see this, for $f\in H^{-1+|\gamma|}(\Sm^1)$ so that $\Upsilon' f\in L^2_\gamma$, let us check that $\L_\gamma \Upsilon' f =0$ (in the distributional sense that ${}^t \L_\gamma \iota_\gamma (\Upsilon' f) = 0$): for $\psi\in \dot{C}^\infty$, 
    \begin{align*}
	\langle {}^t \L_\gamma \iota_\gamma (\Upsilon' f), \psi \rangle = \dprod{\Upsilon' f}{\L_\gamma \psi}_{L^2_\gamma} = \langle f, \neu \L_{\gamma,D}^{-1} \L_\gamma \psi\rangle_{H^{-1+|\gamma|}, H^{1-|\gamma|}} =0,
    \end{align*} 
    since $\psi\in \dom (\L_{\gamma,D})$ and $\neu\psi = 0$. Hence $\Upsilon'$ is a continuous $\dom(\L_{\gamma,max})$-valued map.

    We finally show that $\tilde\tau_\gamma^D \Upsilon' = id|_{H^{-1+|\gamma|}(\Sm^1)}$. For $f\in H^{-1+|\gamma|}(\Sm^1)$ and $h\in H^{1-|\gamma|}(\Sm^1)$, we have
    \begin{align*}
	\langle \tilde\tau_\gamma^D \Upsilon' f, h \rangle_{H^{-1+|\gamma|}, H^{1-|\gamma|}} &\stackrel{\eqref{eq:Dext}}{=} - \dprod{\cancel{\L_\gamma \Upsilon' f}}{Rh}_{L^2_\gamma} + \dprod{\Upsilon' f}{\L_\gamma Rh}_{L^2_\gamma} \\
	&= \langle f, \neu \L_{\gamma,D}^{-1} \L_\gamma Rh \rangle_{H^{-1+|\gamma|}, H^{1-|\gamma|}},
    \end{align*} 
    and since $Rh\in \dom(\L_{\gamma,D})$, $\neu \L_{\gamma,D}^{-1} \L_\gamma Rh = \neu Rh = h$, and hence $\tilde\tau_\gamma^D \Upsilon' f = f$ as elements of $H^{-1+|\gamma|}(\Sm^1)$.

    The proof of Proposition \ref{prop:tauDext} is complete. 
\end{proof}

\subsection{Extension \eqref{eq:G2extmax} of Green's second identity} \label{sec:G2ext}

We first prove Lemma \ref{lem:directsum}.

\begin{proof}[Proof of Lemma \ref{lem:directsum}] With $\L_{\gamma,D}^{-1}$ defined in \eqref{eq:Linv}, given $f\in \dom(\L_{\gamma,max})$, since $(\L_\gamma-\lambda) f\in L^2_\gamma$, we may define $f_D := (\L_{\gamma,D}-\lambda)^{-1} (\L_\gamma-\lambda) f \in \wtH^{2,\gamma}_D(\Dm)\subset \dom(\L_{\gamma,max})$ such that $(\L_\gamma-\lambda) f_D = (\L_{\gamma}-\lambda) f$ and $\tau_\gamma^D f_D =0$, and set $f_\lambda := f-f_D$. Clearly $f_\lambda\in \dom(\L_{\gamma,max})$ and $(\L_\gamma-\lambda) f_\lambda = 0$, i.e. $f_\lambda \in \fN_\lambda (\L_{\gamma,max})$.

    To show uniqueness, if $f\in \dom (\L_{\gamma,D}) \cap \fN_\lambda (\L_{\gamma,max})$, then $(\L_\gamma-\lambda) f = 0$ with $\tau_D f = 0$ which, by injectivity of $\L_{\gamma,D}-\lambda$, implies $f=0$. Lemma \ref{lem:directsum} is proved.    
\end{proof}

We finally prove how to extend Green's second identity to \eqref{eq:G2extmax}. In this paragraph, inner products with no subscripts are implicitly $L^2_\gamma$-inner products. Using Lemma \ref{lem:directsum} with $\lambda=0$, we compute 
\begin{align*}
    \dprod{\L_\gamma f}{g} &- \dprod{f}{\L_\gamma g} \\
    &\stackrel{\eqref{eq:directsum}}{=} \dprod{\L_\gamma f}{g_D} + \dprod{\L_\gamma f}{g_0} - \dprod{f_D}{\L_\gamma g} - \dprod{f_0}{\L_\gamma g} \\
    &\stackrel{\eqref{eq:BTid}}{=} \dprod{f}{\L_\gamma g_D} - \langle \tilde\tau_\gamma^D f, \tau_\gamma^N g\rangle_{-1+|\gamma|,1-|\gamma|} + \dprod{\L_\gamma f}{g_0} \dots\\
    & \qquad - \dprod{\L_\gamma f_D}{g} + \langle \tilde\tau_\gamma^D g,\tau_\gamma^N f_D \rangle_{-1+|\gamma|, 1-|\gamma|} - \dprod{f_0}{\L_\gamma g} \\
    &\stackrel{\eqref{eq:directsum}}{=} \dprod{f_D}{\L_\gamma g_D} + \cancel{\dprod{f_0}{\L_\gamma g_D}} - \langle \tilde\tau_\gamma^D f, \tau_\gamma^N g\rangle_{-1+|\gamma|,1-|\gamma|} + \cancel{\dprod{\L_\gamma f_D}{g_0}} + \dprod{\cancel{\L_\gamma f_0}}{g_0} \dots\\
    & \qquad - \dprod{\L_\gamma f_D}{g_D} - \cancel{\dprod{\L_\gamma f_D}{g_0}} + \langle \tilde\tau_\gamma^D g,\tau_\gamma^N f_D \rangle_{-1+|\gamma|, 1-|\gamma|} - \cancel{\dprod{f_0}{\L_\gamma g_D}} - \dprod{f_0}{\cancel{\L_\gamma g_0}}.
\end{align*}
Finally, $\dprod{f_D}{\L_\gamma g_D} = \dprod{\L_\gamma f_D}{g_D}$ because $(\L_{\gamma}, \wtH^{2,\gamma}_D(\Dm))$ is self-adjoint, and \eqref{eq:G2extmax} follows.

The proof of Theorem \ref{thm:second} is complete.

\appendix

\section{Proofs of auxiliary lemmas} \label{eq:auxlemmas}

\subsection{Characterization of the Friedrichs extension - proof of Lemma \ref{lem:friedrichs-lemma}}\label{app:Friedrichs}

\begin{proof}[Proof of Lemma \ref{lem:friedrichs-lemma}]
    After adjusting $T_0$ by a constant multiple of $I$ (noting that this does not affect any domains involved), we may assume that $T_0\ge 1$.
    The $\alpha$-closure of $D(T_0)$ is precisely the space $V$ described above, i.e.\ the Hilbert space closure of $D(T_0)$ with respect to the norm $u\mapsto(\alpha(u,u))^{1/2}$. Since the domain of the Friedrichs extension is $D_F$ as defined in \eqref{eq:DF}, with $D_F\subset V$, it follows that if $T$ is the Friedrichs extension of $T_0$, then $D(T) = D_F$ must be contained in $V$, i.e.\ the $\alpha$-closure of $D(T_0)$.
    
    Conversely, suppose that $D(T)$ is contained in the $\alpha$-closure of $\text{dom }\alpha$, i.e.\ in $V$. We now claim that, with $\tilde\alpha$ the extension of $\alpha$ to $V$, 
    \begin{align*}
	\tilde\alpha(u,v) = ( Tu,v)_{\H}, \qquad u\in D(T), \qquad v\in V.
    \end{align*}
    We first do so when $v\in D(T_0)$. Since $u\in D(T)\subset V$, it follows that there is a sequence $u_n\in D(T_0)$ such that $u_n\to_{\alpha}u$. In that case, we also have $\tilde\alpha(u,v) = \lim_{n\to\infty}\tilde\alpha(u_n,v) = \lim_{n\to\infty}{\alpha(u_n,v)}$. But we also have
    \begin{align*}
	\lim_{n\to\infty}{\alpha(u_n,v)} = \lim_{n\to\infty}{( T_0u_n,v)_{\H}} = \lim_{n\to\infty}{( u_n,T_0v)_{\H}} = ( u,T_0v)_{\H} = ( u,Tv)_{\H} = ( Tu,v)_{\H}.
    \end{align*}
    The second equality follows from the symmetry of $T_0$. The third equality follows since $u_n\to_{\alpha}u$ necessarily implies $u_n\to u$ in $\H$. The fourth equality follows from $T$ being an extension of $T_0$. Finally, the last equality follows from the symmetry of $T$ on $D(T)$, noting that both $u$ and $v$ belong to $D(T)$. Thus, $\tilde\alpha(u,v) = ( Tu,v)_{\H}$ for $u\in D(T)$ and $v\in\text{dom }\alpha$. To extend the result to $v\in V$, we note that $v\in V$ implies the existence of $v_n\in D(T_0)$ such that $v_n\to_{\alpha}v$. Then
    \[\tilde\alpha(u,v) = \lim_{n\to\infty}{\tilde\alpha(u,v_n)} = \lim_{n\to\infty}{( Tu,v_n)_{\H}} = ( Tu,v)_{\H}.\]
    The second equality follows from the case previously considered, while the third follows since $v_n\to_{\alpha}v$ necessarily implies $v_n\to v$ in $\H$. Hence, we have $\tilde\alpha(u,v) = ( Tu,v)_{\H}$ for $u\in D(T)$ and $v\in V$. It follows, for $u\in D(T)$, that
    \[|\tilde\alpha(u,v)|\le\|Tu\|_{\H}\|v\|_{\H},\]
    i.e., $u\in D_F$. Thus, $D(T)$ is contained in the domain of the Friedrichs extension. Since $T$ is also self-adjoint, it follows that $T$ must be the Friedrichs extension itself, since the domain of a self-adjoint extensions cannot be properly contained in the domain of another self-adjoint extension (this follows by taking adjoints).
\end{proof}

\subsection{Proof of \response{Lemma \ref{real-analysis-lemma}}}\label{sec:RAlemmas}

\response{
\begin{proof}[Proof of Lemma \ref{real-analysis-lemma}]
Let $s_j=\sum_{j=1}^k 1/a_k$; then $s_j\to\infty$ as $j\to\infty$ by assumption. Let
\[c_k^{(j)} = \begin{cases} \frac{1}{s_ja_k} & 1\le k\le j, \\ 0 & k>j\end{cases}.\]
Then
\[\sum_{k=1}^\infty {c_k^{(j)}} = \frac{1}{s_j}\sum_{k=1}^j{\frac{1}{a_k}} = 1\]
by construction, and
\[\sum_k{a_k|c_k^{(j)}|^2} = \sum_{k=1}^j\frac{1}{s_j^2a_k} = \frac{1}{s_j}\to 0\]
since $s_j\to\infty$ by assumption.
\end{proof}
}

\subsection{Proof of Lemma \ref{lem:Cmbounds}}\label{sec:lemCmbounds}
Recall the expression
\begin{align}
    C_{m; s,\gamma} &= \frac{\langle m\rangle^{2s-2\gamma-2}}{\pi (\gamma!)^2} \sum_{\ell\ge 0} (m+2\ell + 1+\gamma)^{1-2s} \frac{(m+\ell+\gamma)!}{(m+\ell)!} \frac{(\ell+\gamma)!}{\ell!}
    \label{eq:Cmsgamma}
\end{align}

We note, for $\gamma\ge 0$ and $s>\gamma+1$, that 
\[1-2s\le 1-2s+2\gamma = -1-2(s-\gamma-1) <-1.\]
Noting the asymptotic $\lim_{x\to\infty}\frac{(x+\gamma)!}{x!(x+1)^{\gamma}} = 1$, it follows that the summand in \eqref{eq:Cmsgamma} is bounded from above and from below by a multiple of $(m+2\ell+1+\gamma)^{1-2s}(m+\ell+1)^{\gamma}(\ell+1)^{\gamma}$. Thus, it suffices to establish upper and lower bounds on the sum
\[\langle m\rangle^{2s-2\gamma-2}\sum_{\ell\ge 0}{(m+2\ell+1+\gamma)^{1-2s}(m+\ell+1)^{\gamma}(\ell+1)^{\gamma}}.\]
For the upper bound, note that
\[m+2\ell+1+\gamma \ge m+\ell+1\implies (m+2\ell+1+\gamma)^{1-2s}\le (m+\ell+1)^{1-2s}\]
since $1-2s<0$, and
\[\ell+1\le m+\ell+1\implies (\ell+1)^{\gamma}\le (m+\ell+1)^{\gamma}\]
since $\gamma>0$. Hence, we can estimate
\begin{align*}
    C_{m;s,\gamma} &\le C\langle m\rangle^{2s-2\gamma-2}\sum_{\ell\ge 0}{(m+2\ell+1+\gamma)^{1-2s}(m+\ell+1)^{\gamma}(\ell+1)^{\gamma}} \\
    &\le C\langle m\rangle^{2s-2\gamma-2}\sum_{\ell\ge 0}{(m+\ell+1)^{1-2s+2\gamma}},
\end{align*}
where $C$ does not depend on $m$. Note that the sum does converge since $1-2s+2\gamma<-1$. Since $x\mapsto(m+x)^{1-2s+2\gamma}$ is decreasing, we have 
\[\sum_{\ell\ge 0}{(m+\ell+1)^{1-2s+2\gamma}}\le \int_0^{\infty}{(m+x)^{1-2s+2\gamma}\,dx} = \frac{m^{2-2s+2\gamma}}{2s-2\gamma-2}.\]
It follows that
\[C_{m;s,\gamma} \le C\langle m\rangle^{2s-2\gamma-2}\frac{m^{2-2s+2\gamma}}{2s-2\gamma-2}\le \frac{C}{2s-2\gamma-2},\]
thus establishing the upper bound. For the lower bound, we assume without loss of generality that $m$ is sufficiently large, say $m>\gamma$, and discard the (positive) terms for $0\le \ell<m$. For the remaining terms, we have
\begin{align*}
 2\ell+2>m+\ell+1 &\implies\ell+1>\frac{1}{2}(m+\ell+1)\implies(\ell+1)^{\gamma}>2^{\gamma}(m+\ell+1)^{\gamma}, \\
m+2\ell+1+\gamma < 2(m+\ell+1) &\implies (m+2\ell+1+\gamma)^{1-2s} > 2^{1-2s}(m+\ell+1)^{1-2s}.
\end{align*}
We conclude that
\begin{align*}
    C_{m;s,\gamma}&\ge c\langle m\rangle^{2s-2\gamma-2}\sum_{\ell\ge m}{(m+2\ell+1+\gamma)^{1-2s}(m+\ell+1)^{\gamma}(\ell+1)^{\gamma}} \\
    &\ge 2^{1-2s+\gamma}c\langle m\rangle^{2s-2\gamma-2}\sum_{\ell\ge m}{(m+\ell+1)^{1-2s+2\gamma}},
\end{align*}
where $c$ does not depend on $m$. Again using that $x\mapsto(m+x)^{1-2s+2\gamma}$ is decreasing, we have
\[\sum_{\ell = m}^\infty {(m+\ell+1)^{1-2s+2\gamma}}\ge\int_{m+1}^{\infty}{(m+x)^{1-2s+2\gamma}\,dx} = \frac{(2m+1)^{2-2s+2\gamma}}{2s-2\gamma-2}.\]
It follows that
\[C_{m;s,\gamma}\ge 2^{1-2s+\gamma}c\langle m\rangle^{2s-2\gamma-2}\frac{(2m+1)^{2-2s+2\gamma}}{2s-2\gamma-2} \ge c'\]
for all $m>\gamma$, where
\[c' = \frac{2^{1-2s+\gamma}c}{2s-2\gamma-2}\left(\inf_{m\ge 0}{\frac{\langle m\rangle}{2m+1}}\right)^{2s-2\gamma-2} = \frac{2^{1-2s+\gamma}5^{1-s+\gamma}c}{2s-2\gamma-2}>0.\]
This establishes the desired lower bound.

\subsection*{Acknowledgements.}

The authors acknowledge partial support from NSF-CAREER grant DMS-1943580. The authors would like to thank Jussi Behrndt who took an interest in the project and provided many helpful comments and references, as well as Rafe Mazzeo, Andr\'as Vasy and Charlie Epstein for providing feedback on an early version of the manuscript. \response{The authors would also like to thank the anonymous referee, whose comments helped improve the clarity of the article.}

\bibliographystyle{siam}

%\bibliography{./bibliography}

\end{document}